\documentclass[11pt, oneside]{article}   	% use "amsart" instead of "article" for AMSLaTeX format
\usepackage[margin=1in]{geometry}
\usepackage{geometry}                		% See geometry.pdf to learn the layout options. There are lots.
\usepackage{graphicx}				% Use pdf, png, jpg, or eps§ with pdflatex; use eps in DVI mode
\usepackage{amssymb}
\usepackage{amsthm}
\usepackage{amsmath}
\usepackage[pdftex=true,hyperindex=true,pdfborder={0 0 0}]{hyperref}
\usepackage[table]{xcolor}

\theoremstyle{plain}
\newtheorem{theorem}{Theorem}[section]
\newtheorem{proposition}[theorem]{Proposition}
\newtheorem{corollary}[theorem]{Corollary}
\newtheorem{conjecture}[theorem]{Conjecture}

\theoremstyle{remark}
\newtheorem{remark}[equation]{Remark}

\def\cyc{\hbox{\textsc{Cyc}}}
\def\seq{\hbox{\textsc{Seq}}}
\def\set{\hbox{\textsc{Set}}}

\usepackage{xcolor}
\newcommand{\oeis}[1]{\color{blue}\textnormal{\href{https://oeis.org/#1}{#1}}\color{black}}	% OEIS

\title{A unified treatment of families of partition funcions}

\author{
	Lida Ahmadi 
	\and
	Ricardo G\'omez A\'iza
	\and
	Mark Daniel Ward
}

\date{}							% Activate to display a given date or no date

\begin{document}
\maketitle

\begin{abstract}
	We present a unified framework of combinatorial descriptions,
        and the analogous asymptotic growth of the coefficients of two
        general families of functions related to integer partitions.
	In particular, we resolve several conjectures and verify several claims
	that are posted on the On-Line Encyclopedia of Integer Sequences.
	We perform the asymptotic analysis by systematically applying
        the Mellin transform, residue analysis, 
	and the saddle point method. 
        The combinatorial descriptions of these families of
        generalized partition functions involve colorings
	of Young tableaux, along with their ``divisor diagrams'',
        denoted with sets of colors whose sizes are controlled by
        divisor functions.
	\medskip

	
	\smallskip
	
\end{abstract}

\renewcommand{\contentsname}{}
%{\footnotesize\tableofcontents}

\section{Introduction}

Given a triple of nonnegative integers $(i,j,k)$ with $i+j+k\geq 1$, we consider the generating functions
\begin{align}
\label{eq:GeneralPartitionMulti}
	P(z) = P^{\left<i,j,k\right>}(z) & = \prod (1-z^{n_1\cdots n_{i}d_{1}\cdots
          d_{j}e_{1}\cdots e_{k}})^{-n_1\cdots n_i/d_{1}\cdots d_{j}}
\end{align}
and
\begin{align}
\label{eq:GeneralPartitionPower}
	Q(z) = Q^{\left<i,j,k\right>}(z) & = \prod (1+z^{n_1\cdots n_{i}d_{1}\cdots
          d_{j}e_{1}\cdots e_{k}})^{n_1\cdots n_i/d_{1}\cdots d_{j}} 
\end{align}
with the $(i+j+k)$-fold product taken over all positive integers $n$'s, $d$'s,
and $e$'s.  
The $n$'s refer to the indices occurring in the exponent as \emph{n}umerators; the $d$'s refer to
indices occurring in the exponent as \emph{d}enominators; and the $e$'s are \emph{e}xtra indices.
We say that any such infinite product represented by the triple $(i,j,k)$ is~\emph{admissible}. 
%(henceforth we refer to the triples of integers $(\ell, h, k)$ satisfying the conditions above as \emph{admissible}).

\subsection{Integer partitions, and integer partitions with distinct summands}
The triple $(i,j,k) = (0, 0, 1)$ yields the most well-known specific
case of this family of generating functions:
\begin{align}
\label{eq:GFPartitions}
	P(z) = \prod\limits_{k=1}^\infty \frac{1}{1-z^k} % \\
	& = \sum\limits_{n=0}^\infty{P_n}z^n % \\
	% & = 1 + z + 2z^2 + 3z^3 + 5z^4 + 7z^5 + 11z^6\ldots
\end{align}
and 
\begin{align}
\label{eq:GFPartitionsDistinct}
	Q(z) = \prod\limits_{k=1}^\infty (1+z^k) % \\
	& = \sum\limits_{n=0}^\infty{Q_n}z^n. % \\
	% & = 1 + z + z^2 + 2z^3 + 2z^4 + 3z^5 + 4z^6 \ldots 
\end{align}
In this case, the generating function $P(z)$ enumerates the combinatorial
class $\mathcal{P}$ of integer partitions, i.e., $P_n$ is the number
of partitions of~$n$.  Similarly, the generating function $Q(z)$
enumerates the combinatorial class 
$\mathcal{Q} \subset \mathcal{P}$ of integer partitions
with distinct summands, or equivalently, $Q_n$ is the number of
partitions of~$n$ with distinct summands.

As unlabelled structures, the product formulas in equations
\eqref{eq:GFPartitions} and \eqref{eq:GFPartitionsDistinct} are
combinatorially justified by viewing
$\mathcal P$ as a multiset of positive integers and $\mathcal Q$
as a powerset of positive integers.  In other words, these classes
can be defined by the well-known combinatorial specifications
\begin{align}
\label{eq:SpecPartitionsMulti}
	\mathcal{P} \triangleq \textsc{MSet}\big(\textsc{Seq}_{\geq 1}(\mathcal Z)\big)
\end{align}
and
\begin{align}
\label{eq:SpecPartitionsPower}
	\mathcal{Q} \triangleq \textsc{PSet}\big(\textsc{Seq}_{\geq 1}(\mathcal Z)\big),
\end{align}
with $\mathcal Z$ denoting the atomic class. In this context, $P(z)$ and $Q(z)$ can be referred to
as \emph{multiset} and \emph{powerset} partition functions, respectively.
(See \cite {FlajoletSedewick09} for a broad general background
on analytic combinatorics.)
The genesis of the study of partition functions goes back to the work of Euler \cite{Euler48}.
The so called ``partition problem'' consists of finding the asymptotic growth of the coefficients of
functions like $P(z)$ and $Q(z)$. The coefficients of the function $P(z)$ in fact result from 
what we now know as \emph{the Euler transform} of the sequence $a(n) = 1$,
which is defined in general as
\begin{align}
	\prod\limits_{n \geq 1} \frac{1}{(1-z^n)^{a(n)}}.
\end{align}
The celebrated and well-known works of Hardy and Ramanujan
are considered the first main contributions to the partition problem
(see \cite{HardyRamanujan18A, HardyRamanujan18B}). 
In particular,
for the triple $(i,j,k) = (0, 0, 1)$, they showed that
\begin{align}
\label{eq:HardyRam}
	P_n \sim \frac{e^{\pi\sqrt{2n/3}}}{4\sqrt{3}\!\ n}
\qquad\hbox{and}\qquad
	Q_n \sim \frac{e^{\pi\sqrt{n/3}}}{4\cdot3^{1/4}\!\ n^{3/4}}.
\end{align}
The main method they used is nowadays known as the ``circle method'' and is based
on the representation of the coefficients of power series by means of the integral Cauchy formula.
One of our goals is to solve the partition problem (first order asymptotics) for the more
general class of partition functions of the forms given in
equations \eqref{eq:GeneralPartitionMulti} and \eqref{eq:GeneralPartitionPower}.
\begin{table}[h!]
\centering
\footnotesize
	\begin{tabular}{|c || c | c|}
	\hline
		$(i,j,k)$ & $\log ([z^n]P(z))\sim$ & $\log ([z^n]Q(z))\sim$ \\
		\hline\hline
		$(1,0,0)$ & \cellcolor{gray!25} given in OEIS \oeis{A000219}
			&  \cellcolor{gray!25} given in OEIS \oeis{A026007} \\
		\hline
		$(2,0,0)$ & \cellcolor{gray!25} given in OEIS \oeis{A280540}
			&  \cellcolor{blue!25} conjectured in OEIS \oeis{A280541} \\
		\hline
		$(3,0,0)$ & \cellcolor{blue!25} conjectured in OEIS \oeis{A318413}
			& \cellcolor{blue!25} conjectured in OEIS \oeis{A318414} \\
		\hline
		$(1,0,1)$ & \cellcolor{gray!25} given in OEIS \oeis{A061256}
			& \cellcolor{gray!25} given in OEIS \oeis{A192065} \\
		\hline
		$(1,0,2)$ & \cellcolor{green!25} not given in OEIS \oeis{A174467}
			& $\times$ \\
		\hline \hline
		$(0,0,1)$ & \cellcolor{gray!25} given in OEIS \oeis{A000041}
			& \cellcolor{gray!25} given in OEIS \oeis{A000009}  \\
		\hline
		$(0,0,2)$ & \cellcolor{gray!25} given in OEIS \oeis{A006171}
			& \cellcolor{blue!25} conjectured in OEIS \oeis{A107742} \\
		\hline
		$(0,0,3)$ & \cellcolor{green!25} not given in OEIS \oeis{A174465}
			& \cellcolor{green!25} not given in OEIS \oeis{A280473} \\
		\hline
		$(0,0,4)$ & \cellcolor{green!25} not given in OEIS \oeis{A280487}
			& \cellcolor{green!25} not given in OEIS \oeis{A280486} \\
			\hline
		$(0,1,1)$ & \cellcolor{gray!25} given in OEIS \oeis{A305127}
			& \cellcolor{gray!25} given in OEIS \oeis{A318769} \\
		\hline \hline
		$(0,1,0)$ & \cellcolor{red!25} \emph{incorrect in OEIS} \oeis{A028342}
			& \cellcolor{blue!25} conjectured in OEIS \oeis{A168243} \\
		\hline
		$(0,2,0)$ & \cellcolor{green!25} not given in OEIS \oeis{A318695}
			& \cellcolor{green!25} not given in OEIS \oeis{A318696} \\
		\hline
		$(0,3,0)$ & \cellcolor{green!25} not given in OEIS \oeis{A318966}
			& \cellcolor{green!25} not given in OEIS \oeis{A318967} \\
		\hline
%		\hline
%		$(0,4,0)$ & $\times$ & \color{blue}$XXX$ \color{black}
%			& $\times$ & \color{blue}$XXX$ \color{black} \\
%		\hline
%		$(4,0,0)$ & $\times$ & \color{blue}$XXX$ \color{black}
%			& $\times$ & \color{blue}$XXX$ \color{black} \\
%		\hline
		% $(4,1,0)$ & $\times$ & $\times$ \\
		% $(4,2,0)$ & $\times$ & $\times$ \\
		% $(4,3,0)$ & $\times$ & $\times$ \\
		% $(4,4,0)$ & $\times$ & $\times$ \\
		% $(4,0,1)$ & $\times$ & $\times$ \\
		% $(4,0,2)$ & $\times$ & $\times$ \\
		% $(4,0,3)$ & $\times$ & $\times$ \\
		% $(4,0,4)$ & $\times$ & $\times$ \\
		% $(4,1,1)$ & $\times$ & $\times$ \\
		% $(4,2,1)$ & $\times$ & $\times$ \\
		% $(4,3,1)$ & $\times$ & $\times$ \\
		% $(4,1,2)$ & $\times$ & $\times$ \\
		% $(4,2,2)$ & $\times$ & $\times$ \\
		% $(4,1,3)$ & $\times$ & $\times$ \\
		% $(5,0,0)$ & $\times$ & $\times$ \\
%		\vdots & \vdots & \vdots & \vdots & \vdots \\
		% \hline
	\end{tabular}
	\caption{A list of 25 sequences and the analogous OEIS entries
          corresponding to the 
          generating functions in
          equations~\eqref{eq:GeneralPartitionMulti}
          and~\eqref{eq:GeneralPartitionPower}. 
These entries are collected into three groups: $(i \geq 1, j, k)$, $(0, j, k\geq 1)$ and $(0, j\geq 1, 0)$.
	The asymptotic growth
of $\log ([z^n]P(z))$ or $\log ([z^n]Q(z))$ is either (a) given in
OEIS; (b) conjectured in OEIS but incorrect; (c) stated only \emph{as a
        conjecture} in OEIS; or (d) not given in OEIS. %(9 items).
	%  ($\times$) No OEIS registries were found.
	}
	\label{tab:MultiPowerOEIS}
\end{table}

\begin{table}[h!]
\centering
\footnotesize
	\begin{tabular}{|c || c | c|}
	\hline
		$(i,j,k)$ & $[z^n]P(z)\sim$ & $[z^n]Q(z)\sim$ \\
		\hline\hline
		$(1,0,0)$ & \cellcolor{gray!25} given in OEIS \oeis{A000219}
			&  \cellcolor{gray!25} given in OEIS \oeis{A026007} \\
		\hline
		$(1,0,1)$ & \cellcolor{gray!25} given in OEIS \oeis{A061256}
			& \cellcolor{gray!25} given in OEIS \oeis{A192065} \\
		\hline \hline
		$(0,0,1)$ & \cellcolor{gray!25} given in OEIS \oeis{A000041}
			& \cellcolor{gray!25} given in OEIS \oeis{A000009}  \\
		\hline \hline
		$(0,1,0)$ & \cellcolor{green!25} not given in OEIS \oeis{A028342}
			& \cellcolor{green!25} not given in OEIS \oeis{A168243} \\
		\hline
		$(0,2,0)$ & $\times$
			& \cellcolor{green!25} not given in OEIS \oeis{A318696} \\
		\hline
	\end{tabular}
	\caption{Status in OEIS of the asymptotic behavior of the coefficient
	  for the cases in Corollary \ref{cor:saddle1P}.}
	\label{tab:asymptoticsOEIS}
\end{table}

\subsection{A unified framework}

Several other particular cases have also been previously considered by
other authors.  We endeavor to introduce a framework to
systematically classify these families of combinatorial classes and
their analogous generating functions.  Table~\ref{tab:MultiPowerOEIS} shows
a list of entries that appear in the On-Line Encyclopedia of Integer Sequences (OEIS)
\cite{Sloane06} that correspond to certain triples $(i,j,k)$. We will see that, in general,
what we are dealing with now is what we called the \emph{``cyclic'' Euler transform} of a sequence $a(n)$
which we now define as the sequence of coefficients of the function
\begin{align}
	\prod_{n\geq 1} \exp\left( \sum\limits_{\ell=1}^\infty \frac{z^{n \ell }}{n\ell} \right)^{a(n)}.
\end{align}

Among the cases that we identified in the OEIS, some contain an
asymptotic analysis, and some give a combinatorial description, but we
endeavor to handle all such cases systematically.
Several claims about the asymptotic growth of the coefficients of the
generating functions are stated only as conjectures without proofs, all due to
V. Kot\u{e}\u{s}ovec. 
We aim to present a more general treatment and a unifying framework for
the combinatorial specifications of the labelled classes,
as well as a systematic study of the asymptotic analysis of the number
of these combinatorial objects.

\subsection{On Kot\u{e}\u{s}ovec's conjectures}

We are able to resolve all of Kot\u{e}\u{s}ovec's conjectures about
these families of combinatorial classes and their generating functions. 
We are delighted to discover that all of his conjectures are true, 
with one exception:  In the case of the admissible triple $(i,j,k) = (0,1,0)$, 
Kot\u{e}\u{s}ovec's conjecture about the asymptotic growth of the
logarithm of coefficients is incorrect.  
We provide proofs of all of his conjectures.  In the case of the
erroneous conjecture about $(0,1,0)$, we first explain the likely
source of the confusion (likely due to numerical error), and we
provide the corrected asymptotic analysis for this case as well.

Because the erroneous conjecture led us to study this family of
combinatorial objects, generating functions, and their asymptotic
analysis, we start with Kot\u{e}\u{s}ovec's slightly erroneous
conjecture about the case~$(0,1,0)$:

\begin{conjecture}[Kot\u{e}\u{s}ovec, Sep. 2018, see \oeis{A028342}\ in OEIS]
\label{conj:Kotesovec}
	Let
	\begin{equation}
	\label{eq:GFKotesovecWrong}
		P(z) \triangleq \prod_{d=1}^\infty(1-z^d)^{-1/d}
		= \sum_{n=0}^\infty \frac{p_n}{n!}z^n.
	\end{equation}
	The coefficients of $P(z)$ have the asymptotic growth rate
	\begin{equation}
	\label{eq:KotesovecConjP}
		\log \frac{p_n}{n!} \sim \frac{\log 2}{2}\log^2 n.
	\end{equation}
\end{conjecture}

In order to explain the correct asymptotic behavior, and the likely
source of the erroneous asymptotics, we need some well-known constants.
For every $n \geq 0$, let $\gamma_n$ denote the $n$th Stieltjes
number.  In particular, let $\gamma \triangleq \gamma_0$
be the Euler-Mascheroni constant. Also, let $\operatorname{W}(z)$ denote the
Lambert $\operatorname{W}$ function. 
Theorem~\ref{thm:KotesovecP} is the corrected version of
Conjecture~\ref{conj:Kotesovec}.  It is also an exemplary prototype of
the style of asymptotic results that can be deduced from our analysis.

\begin{theorem}[Combinatorial description and asymptotic behavior for $(0,1,0)$, see \oeis{A028342}]
\label{thm:KotesovecP}
	Let $P(z)$ and $\{p_n\}_{n\geq 0}$ be given as in equation \eqref{eq:GFKotesovecWrong}.
	\begin{itemize}
	\item
		\textsc{(Combinatorial description)} The sequence
                $\{p_n\}_{n\geq 0}$ enumerates 
		the class of colored permutations by number of divisors, that is, permutations that after decomposing
		them into a product of cycles, each cycle carries a label
		that is a divisor of its corresponding length (see \S \ref{subsubsec:010}).
%                {\color{red} We probably want to put a reference here
%                  to the color scheme in the paper; otherwise, the
%                  reader will not know what we are talking about, when
%                  we say ``class of colored permutations by number of divisors.''}
	\item
		\textsc{(Asymptotic behavior)} Furthermore, 
		let $c \triangleq \pi^2/12 - \gamma^2/2-2\gamma_1$ and $w_n \triangleq \operatorname{W}(e^{\gamma}n)$.
		Then (see Theorem \ref{thm:010P})
		\begin{align}	
		\label{eq:KotesovecP}
			\frac{p_n}{n!} & \sim
			\frac{(w_n/n)^{1-\gamma}}{\sqrt{2\pi\log (w_n/n)}}
			\exp\left(c+w_n + \frac{\log^2 (w_n/n)}{2}\right) . 
		\end{align}
		In particular,
		\begin{equation}
		\label{eq:KotesovecPLog}
			\log \frac{p_n}{n!} \sim \frac{\log^2n}{2}.
		\end{equation}
	\end{itemize}
\end{theorem}

Let us discuss a possible explanation why
Kot\u{e}\u{s}ovec made his conjecture in equation~\eqref{eq:KotesovecConjP},
which differs by a factor of $\log 2$ from the true asymptotic
behavior given in equation~\eqref{eq:KotesovecPLog}. First, consider
Figure~\ref{fig:KotesovecP}, which illustrates the three
different estimates of $\log(p_n/n!)$ given by:
\begin{itemize}
\itemsep0em
\item[(a)] Kot\u{e}\u{s}ovec's
conjecture in equation~(\ref{eq:KotesovecConjP}),
\item[(b)] the logarithm of the very accurate estimate from equation~(\ref{eq:KotesovecP}), and
\item[(c)] the accurate (but much slower to converge)
estimate from equation~(\ref{eq:KotesovecPLog}).
\end{itemize}
\begin{figure}[h]%tbp] %  figure placement: here, top, bottom, or page
   \centering
   \includegraphics[width=6.5in]{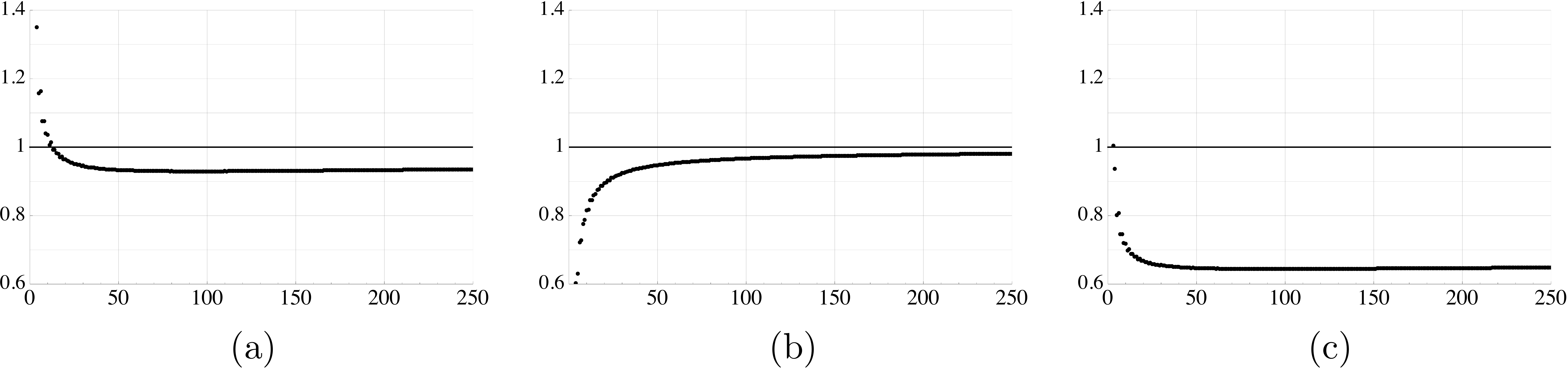} 
   \caption{The three different estimates of $\log p_n/n!$, given by
   (a) Kot\u{e}\u{s}ovec's conjecture in equation~\eqref{eq:KotesovecConjP},
   (b) the logarithm of the very accurate estimate from of equation~\eqref{eq:KotesovecP}, and
   (c) the accurate (but much slower to converge)
estimate from equation~\eqref{eq:KotesovecPLog}.}
   \label{fig:KotesovecP}
\end{figure}
In the illustration, the central item (b) shows that equation~\eqref{eq:KotesovecP}
gives a very accurate estimation. Furthermore, at first sight, the left item (a) shows
that the estimation given by equation~\eqref{eq:KotesovecConjP} is intuitively compatible with
Kot\u{e}\u{s}ovec's conjecture \ref{conj:Kotesovec}, whereas the right item (c) shows that
the estimation with equation~\eqref{eq:KotesovecPLog} is slow to
converge (but, perhaps surprisingly, will eventually converge to~1).

Computationally, obtaining the exact numerical values of the coefficients of $P(z)$ is
very challenging for indices of order greater that $10^3$.
Indeed, the entry \oeis{A028342}\
in \cite{Sloane06} has a table 
(attributed to Kot\u{e}\u{s}ovec) that lists only the first 455 values of $a_n$.
For these reasons, it is likely that Kot\u{e}\u{s}ovec simply based his
asymptotic estimate on a relatively small number of terms.
It is also likely that he was not aware
of the precise estimation given by equation~\eqref{eq:KotesovecP}.

We leverage the fact that the numeric value of $w_n$ can be obtained for very
large values of $n$.  It is easy to see 
that the logarithm of equation~\eqref{eq:KotesovecP} is
asymptotically equivalent to $w_n^2/2$, and the later is asymptotically
equivalent to $(\log ^2 n)/2$; however, this asymptotic convergence is
extremely slow.  In Table~\ref{tab:estimationSmall},
we illustrate values of $w_n^2/\log^2n$ for indices
$n \leq 10^3$, and we emphasize that (in this range) they mainly
exhibit a decreasing behavior, so it is easy to be misled about their
asymptotic growth.
\begin{table}[h!]
\centering
	\footnotesize
	\begin{tabular}{c|cccccccccc}
		$n$ & $2$ & $3$ & $4$ & $6$ & $8$ & $10$ & $20$ & $50$ & $10^2$ & $10 ^3$\\
		\hline
		$w_n^2/\log^2n$ & $2.7032$ \ & $1.5433$ \ & $1.2260$ \ & $0.9957$ \ & $0.9027$ \
		& $0.8522$ \ & $0.7605$ \ & $0.7100$ \ & $0.6944$ \ & $0.6899$\\
	\end{tabular}
	\caption{List of values of $w_n^2/\log^2n$ for some $n \leq 10^3$.}
	\label{tab:estimationSmall}
\end{table}

If we divide the last entry in this table by $\log 2$, then 
we obtain $\frac{w_{1000}^2}{(\log2) \log^2 1000} = 0.9954$,
%a number that is close enough to~$1$ for intuitive purposes,
which may have led Kot\u{e}\u{s}ovec to mistakenly 
include $\log{2}$ in his conjecture in
equation~\eqref{eq:KotesovecConjP}, based on numerical evidence.\footnote{His mathematical analysis in publications in the arXiv, like~\cite{Kotesovec16}, 
have inspired colleagues to find theoretical justifications for his
estimates; see, for instance,~\cite{HanXiong18}.}
However, in Table~\ref{tab:estimationBig} we can look at the
same values of $w_n^2/\log^2n$, for much larger indices
$n\geq 10^4$.  This allows us to see the slow asymptotic convergence
that justifies that both estimates in Theorem~\ref{thm:KotesovecP} are accurate.
Taken together with our proof (see Theorem \ref{thm:010P}),
we confirm our numerical estimate of Theorem~\ref{thm:KotesovecP}.
\begin{table}[h!]
\centering
	\footnotesize
	\begin{tabular}{c|cccccccccc}
		$n$ & $10^4$ & $10^{6}$ & $10^{8}$ & $10^{10}$ & $10^{20}$
		& $10^{50}$ & $10^{10^2}$ & $10^{10^3}$ & $10^{10^4}$ & $10^{10^5}$\\
		\hline
		$w_n^2/\log^2n$ & $0.7063$ \ & $0.7437$ \ & $0.7745$ \ & $0.7987$ \
		 & $0.8666$ \ & $0.9295$ \ & $0.9583$ \ & $0.9937$ \ & $0.9991$ \ & $0.9998$\\
	\end{tabular}
	\caption{List of values of $w_n^2/\log^2n$ for some $n \geq 10^4$.}
	\label{tab:estimationBig}
\end{table}

The rest of the paper is organized as follows.
In Section~\ref{sec:background} we present a brief background on the study
of partition functions. %, mostly pointing out works that are closest to our context.
In Section~\ref{sec:combinatorial} we describe, as \emph{labeled classes},
the combinatorial structures enumerated by the (exponential) generating functions
\eqref{eq:GeneralPartitionMulti} and \eqref{eq:GeneralPartitionPower}.
In Section~\ref{sec:asymptotic}, we discuss the asymptotic analysis of
these generating functions:
First, we compute the Mellin transforms, which can be derived in a
systematic way across the $(i,j,k)$ families. In contrast, the residue
analysis needs to be customized for each triple parameter, though some
generalities can be developed.
The final asymptotic growth is obtained through saddle point method.
With this methodology, we can resolve the conjectures about the
asymptotic growth of the 6 sequences marked with~(b) or~(c) in
Table~\ref{tab:MultiPowerOEIS}.  Additionally, we discover the
asymptotic growth of the 9 sequences in OEIS that are marked with~(d) 
in Table~\ref{tab:MultiPowerOEIS}, and we provide the proofs.
%in the entries
%\oeis{A028342}, \oeis{A318413},
%\oeis{A168243}, \oeis{A107742},
%\oeis{A280541}\ and \oeis{A318414}\
%in OEIS \cite{Sloane06}.

\section{Background}
\label{sec:background}

\subsection{The partition problem}

Ever since the work of Hardy and Ramanujan,
there has been a continued and intensive development
of the study of partition functions, there is a vast and diverse literature
on the subject, the standard general reference is
the book of Andrews \cite{Andrews84} from 1984 and the references
therein. The area has continued developing and
we can only mention a brief list of some more recent
works that are related, up to certain extends, to our context, like
\cite{Kane06, Andrews08, BringmannMahlburg11, BringmannMahlburg13, Andrews13, BringmannMahlburg14, JoKim15}, works with more probabilistic points of view include \cite{CanfieldCorteelHitczenko01, GohHitczenko07, BodiniFusyPivoteau10, deSalvoPak19, KaneRhoades19}, connexions with cellular automata can be found in \cite{HolroydLiggetRomik03},
for other similar frameworks with $q$-series see \cite{BringmannMahlburg11}, and so forth.
There are several ways to generalizations, and in the literature several kinds of classes of functions
are so called ``weighted partitions''. For example, one kind arises by locating
a sequence $(b_n)_{n=1}^\infty$ of ``weights'' as \emph{coefficients},
namely, products of the form
\begin{align}
\label{eq:WeightedPartitionA}
	P(z)=\prod\limits_{n\geq1}\frac{1}{1-b_nz^n}
	\qquad \textnormal{ and } \qquad
	Q(z)=\prod\limits_{n\geq1}(1+b_nz^n)
\end{align}
(see e.g. \cite{Wright34, BringmannMahlburg13, BringmannMahlburg14, Rolen17, Chern18, KimKimLovejoy20, KimKimLovejoy21}).
Another kind arises by locating a sequence $(a_n)_{n\geq 1}$ of ``weights'' as
\emph{exponents}, namely, products of the form
\begin{align}
\label{eq:WeightedPartitionB}
	P(z)=\prod\limits_{n\geq1}\left(\frac{1}{1-z^n}\right)^{a_n}
	\qquad \textnormal{ and } \qquad
	Q(z)=\prod\limits_{n\geq1}(1+z^n)^{a_n}
\end{align}
(see e.g. \cite{Wright31, Wright34-2, Brigham50, Meinardus54, Yang00, GohHitczenko07, MadritschWagner10, HanXiong18, Mutafchiev18}).
A third kind could arise when putting the previous two kinds together, namely,
products of the form
\begin{align}
\label{eq:WeightedPartitionB}
	P(z)=\prod\limits_{n\geq1}\left(\frac{1}{1-b_nz^n}\right)^{a_n}
	\qquad \textnormal{ and } \qquad
	Q(z)=\prod\limits_{n\geq1}(1+b_nz^n)^{a_n},
\end{align}
etc.
In this paper, products of the second kind arise when $j=0$ (see Theorem \ref{thm:ClassPOrdi}),
and also when the triple is $(i,j,k) = (0,1,0)$ (see Theorem \ref{thm:ClassPExpo}),
thus, in these cases there are certain sequences of weights as exponents.
There have been several works that, under some general hypothesis,
resolve the partition problem for families of sequences of weights as exponents
(at least the first order asymptotic growth, and mostly focused, in fact,
in the $\mathcal P$-form rather than the $\mathcal Q$-form).
Probably the most famous result for sequences of weights as exponents is Meinardus' Theorem
\cite{Meinardus54} from 1954, which is also explained in depth in
Theorem 6.2 of \cite[Chapter~6, pp.~88]{Andrews84}.
Even before Meinardus, there are other works, such as Brigham \cite{Brigham50}
in 1950, that consider weighted partitions as a general framework 
for many types of partition problems.  For example, he mentions 
well-known works of Wright \cite{Wright31, Wright34-2} from 1931 and 1934
as particular cases. Meinardus' Theorem has been generalized, as in
\cite{GranovskyStarkErlihson08}, see also \cite{MadritschWagner10},
and for further works in this context see e.g. \cite{BellBurris06, Yang00, Yang01}.
Nevertheless, the immense generalities make many cases to fall beyond the assumptions
of well-known results like Meinardus' Theorem and its generalizations,
and some open conjectures on several
partition problems that have been considered remain,
as far as we know, unsolved, for instance see
\oeis{A028342} \ and \oeis{A318413}\ in OEIS \cite{Sloane06}.
For example, in the case of the former, which is when
$(i,j,k) = (0,1,0)$, what fails in Meinardus' Theorem is the fact that
a pole of $\zeta(s+1)$ is located at $0$.
In the case of the later, or more generally, when the triple is $(i,0,k)$, then
it is necessary to analyze the Dirichlet series $\sum_{n\geq 1}\frac{\psi(n)}{n^s}$,
where $\psi(n)$ is defined in Theorem \ref{thm:ClassPExpo}.  The
Meinardus style of argument is designed for the case of partitions of
integers, that is, when the triple is $(i,j,k) = (0,0,1)$
(for this case, see again \cite[Chapter~6]{Andrews84}).

Furthermore, several other instances of open conjectures that fall
within the scope of equations~\eqref{eq:GeneralPartitionMulti} 
and \eqref{eq:GeneralPartitionPower}
are found in \oeis{A168243}, \oeis{A107742}, \oeis{A280541},
\oeis{A318414}, and we will present their resolutions.
In virtue of Theorem \ref{thm:ClassPExpo}, the general form of the ``weighted partition function'' that we are
considering here has the form
\begin{align}
\label{eq:generalForm}
	\prod\limits_{n\geq 1}\exp\left(\frac{1}{n}\log\frac{1}{1-z^n}\right)^{a_n} 
\end{align}
(the sequence $(a_n)$ that yields equations \eqref{eq:GFPartitions} and \eqref{eq:GFPartitionsDistinct}
is $a_n=n$).
Thus one of our goals here is to bring a source that resolves
the conjectures regarding the partition problem in all these cases,
in addition to verifying all claims.

\subsection{Divisor functions}

Recall that the well-known \emph{sum of divisors} is the divisor function defined for every $n \geq 1$ by
\begin{align}
	\sigma_x(n) := \sum\limits_{d|n}d^x.
\end{align}
For example, $\sigma_0(n)$ is the \emph{number of divisors} of $n$
and $\sigma(n) := \sigma_1(n)$ is the \emph{sigma function}, i.e. the sum of the positive divosors of $n$.
Also, for every $k\geq 1$, the $k$-\emph{fold number of divisors} is the divisor function
defined %\footnote{$\tau_{k}(n)$ is also simply known as the \emph{divisor} function, though here we will not use this convention to avoid any confusion with $\sigma_x(n)$.}
 by
\begin{align}
	\tau_{k}(n) := \sum\limits_{d_1\cdots d_k=n} 1.
\end{align}
In particular, $\tau_1(n)= 1$ and $\tau(n):=\tau_2(n)=\sigma_0(n)$ (we let $\tau_0(n):=1$).
Observe that
\begin{align}
	\tau_{k+1}(n) := \sum_{d|n} \tau_{k}(n/d).
\end{align}
%We will also need certain generalized versions of $\sigma_x(n)$.
%For every $n$ and $k\geq 1$, let
%\begin{align}
%	\sigma_x^{\left<k\right>}(n) := \sum\limits_{d_1\cdots d_k|n} (d_1\dots d_k)^x
%\end{align}
%and also let $\sigma^{\left<k\right>}(n) = \sigma_1^{\left<k\right>} (n)$.
%For example, $\sigma_x^{\left<1\right>}(n) = \sigma_x(n)$.
%More generally, for every non-negative integer  $r \leq k$, let
%\begin{align}
%	\sigma_{x,y}^{\left<k,r\right>}(n) := \sum\limits_{d_1\cdots d_k|n} (d_1\dots d_r)^x(d_{r+1}\dots d_k)^y.
%\end{align}
%Thus, for example, $\sigma_{x,y}^{\left<k,k\right>}(n) = \sigma_{y,x}^{\left<k,0\right>}(n) = \sigma_x^{\left<k\right>}(n)$.
Divisor functions are central in analytic number theory,
in particular because of the well-known identity
\begin{align}
	\zeta(s)^k = \sum\limits_{n\geq 1} \frac{\tau_k(n)}{n^s},
\end{align}
i.e., $\zeta(s)^k$ is the Dirichlet series of $\big(\tau_k(n)\big)_{n\geq 1}$.
For some recent developments see
\cite{Nguyen21, DrappeauTopacogullari19, Blomer18, Lester16}.

\section{Combinatorial specifications}
\label{sec:combinatorial}

In this section we present the general combinatorial specifications that
yield the exponential generating functions given in equations \eqref{eq:GeneralPartitionMulti}
and \eqref{eq:GeneralPartitionPower}. We will focus on the combinatorial class $\mathcal P$
and at the end of the section we will discuss the corresponding class $\mathcal Q$.
When $j=0$ (i.e. there are no denominators), we
get two alternative specifications that yield the same exponential generating functions,
hence we obtain two isomorphic kinds of labeled combinatorial classes.
We will describe several particular examples. We will be interested in coloring several objects
with a countable set of different colors, as depicted in Figure \ref{fig:colors}.
But first let us start with descriptions of
certain basic classes in a way that will be useful throughout. 
\begin{figure}[h] %  figure placement: here, top, bottom, or page
   \centering
   \includegraphics[width=1.3in]{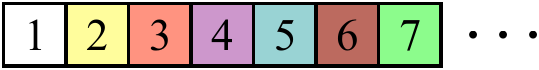} 
   \caption{A countable list of different colors, labelled by $\{1, 2, 3, 4, 5, 6, 7, \ldots\}$.}
   \label{fig:colors}
\end{figure}
\subsection{Permutations, integer partitions, Young tableaux, and divisor diagrams}

The combinatorial class $\mathcal S$ of all permutations can be specified as sets of labelled cycles, i.e.
\begin{align}
\label{eq:permutationSpec}
	\mathcal S(\mathcal Z) = \set \big( \cyc(\mathcal Z)\big),
\end{align}
which translates into the exponential generating function
$$
	\exp\left(\log\frac{1}{1-z}\right) = \frac{1}{1-z}.
$$
For example, a permutation like
\setlength{\tabcolsep}{0.1em}
\begin{align}
%\footnotesize
\label{eq:permutationEx}
\left(\begin{tabular}{cccccccccccccccccccccccccccccccccc}
	1 & 2 & 3 & 4 & 5 & 6 & 7 & 8 & 9 & 10 & 11 & 12 & 13 & 14 & 15 & 16 &
	17 & 18 & 19 & 20 & 21 & 22 & 23 & 24 & 25 & 26 & 27 & 28 & 29 & 30 &
	31 & 32 & 33 & 34\\
	5 & 10 & 19 & 18 & 27 & 8 & 12 & 32 & 15 & 16 & 9 & 31 & 34 & 11 & 29
	& 13 & 17 & 22 & 30 & 2 & 14 & 23 & 4 & 26 & 28 & 33 & 1 & 24 & 6 & 21
	& 7 & 3 & 25 & 20
\end{tabular}\right)
\end{align}
decomposes as a set of cycles coming from its decomposition as a product of cycles, namely,
\begin{align}
	(1 \thinspace 5\thinspace 27)
	(2 \thinspace 10\thinspace 16\thinspace 13\thinspace 34\thinspace 20)
	(3 \thinspace 19\thinspace 30\thinspace 21\thinspace 14\thinspace
	11\thinspace 9\thinspace 15\thinspace 29\thinspace 6\thinspace 8\thinspace 32)
	(4\thinspace 18\thinspace 22\thinspace 23)
	(7\thinspace 12\thinspace 31)
	(17)
	(24\thinspace 26\thinspace 33\thinspace 25\thinspace 28).
\end{align}
In general, we write cycles starting with the ``leader'', and 
we can always ``re-accommodate'' the product of cycles monotonically in a unique way
according to their lengths and, within cycles of equal length, according to the leader.
For example,
\begin{align}
\label{eq:reacommodate}
	(3 \thinspace 19\thinspace 30\thinspace 21\thinspace 14\thinspace
	11\thinspace 9\thinspace 15\thinspace 29\thinspace 6\thinspace 8\thinspace 32)
	(2 \thinspace 10\thinspace 16\thinspace 13\thinspace 34\thinspace 20)
	(24 \thinspace 26\thinspace 33\thinspace 25\thinspace 28)
	(4 \thinspace 18\thinspace 22\thinspace 23)
	(1 \thinspace 5\thinspace 27)	
	(7 \thinspace 12\thinspace 31)
	(17).
\end{align}
Any permutation is associated to an integer partition of its size, namely
the lengths of its cycles, e.g. for the permutation above we have the partition
$34 = 12 + 6 + 5 + 4 + 3 + 3 + 1$. It is customary to represent integer partitions
as \emph{Young diagrams}, i.e. as monotone arrangements of horizontal blocks formed
by consecutive boxes with lengths according to the summands in the partition.
See item (a) in Figure \ref{fig:Young} for the example above
(observe that, for reference, the associated partition
is written to the left of the Young diagram).
\begin{figure}[h] %  figure placement: here, top, bottom, or page
   \centering
   \includegraphics[width=6.5in]{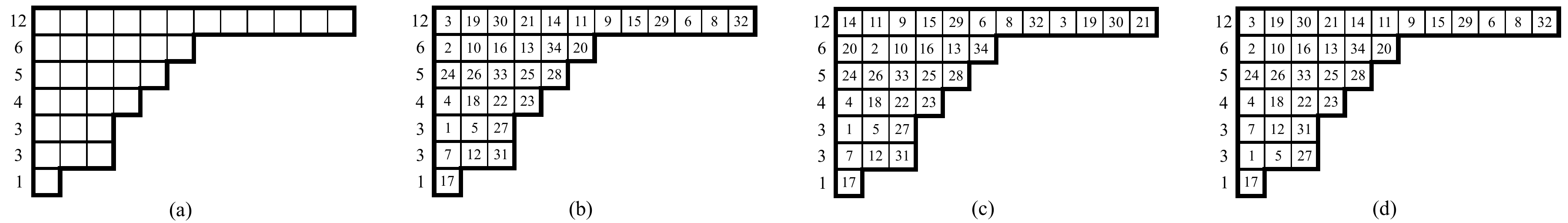} 
   \caption{Young Tableaux: (a) A Young diagram. (b) A Young tableau.
   (c) A Young tableau obtained from (b) by cyclically shifting some rows.
   (d) A Young tableau obtained from (b) by switching rows of the same size.}
   \label{fig:Young}
\end{figure}
A \emph{Young tableau} on a Young diagram of size $n$ is a \emph{well-labelling}
of the boxes of the Young diagram with the set $\{1, \ldots , n\}$ (then, 
all the $n$ symbols are required to be used exactly once).
For example, (b) in Figure \ref{fig:Young} is the unique representation
(described before)
of the partition given as in equation \eqref{eq:reacommodate}.
In fact, every Young tableau has an associated permutation which has
the labelled rows as its decomposition in cycles.
Thus, different labelling of Young diagrams can yield Young tableaux that have
the same associated permutation, for example in Figure \ref{fig:Young},
items (b) and (c) differ only because some rows were shifted cyclically,
that is, the permutation remains invariant under the action $\alpha$
by cyclic shifts, that is, $\alpha(a_1,\dots , a_n) = (a_2, \ldots a_{n}, a_1)$ for every
$(a_1, \ldots,a_n) \in \mathcal A^n$.
A class of Young tableaux is \emph{modular} if
its elements are regarded as equivalence classes of the equivalence
relation defined by the $\alpha$-orbits.
Similarly, in Figure \ref{fig:Young},
items (b) and (d) differ only by a permutation of the two rows of length three,
the other generic instance that make two different Young tableaux induce the same permutation.
Thus there is also an action $\sigma$ that for every $n\geq 1$ acts by permutations 
on the set of $n$-blocks, and the original permutation is also invariant
under this $\sigma$-action. A class of Young tableaux is \emph{symmetric} if its elements are regarded
as equivalence classes of the orbital equivalence relation induced by the $\sigma$-action.
Thus for example, the class $\mathcal S$ of all permutations is represented by the class of
modular \emph{and} symmetric Young tableaux.

The \emph{divisor diagram} associated to a Young diagram is
again an arrangement of blocks of boxes arranged as rows beside
the Young diagram in a way that blocks of length $n$ in the Young
diagram correspond to blocks of length $\tau(n)$ in the divisor diagram
(it is like the ``image'' of the Young diagram under the number of divisor function $\tau(n)$,
thus, in general, the arrangement in the divisor diagram is not monotone).
See Figure \ref{fig:YoungDivisorDiagrams}.
\begin{figure}[htbp] %  figure placement: here, top, bottom, or page
   \centering
   \includegraphics[width=2.8in]{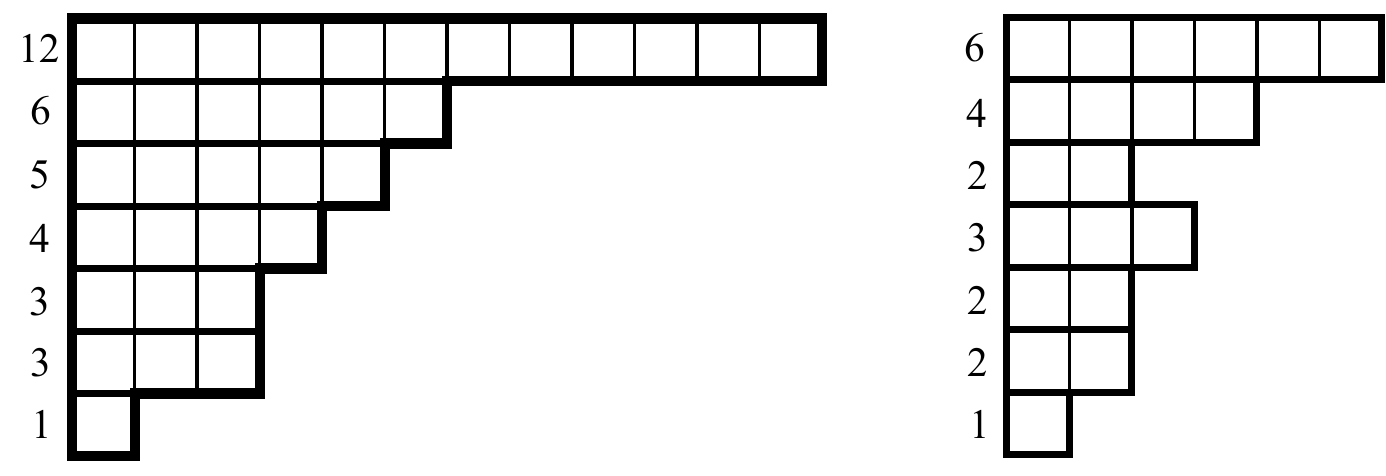} 
   \caption{A Young diagram together with its divisor diagram.}
   \label{fig:YoungDivisorDiagrams}
\end{figure}

Henceforth,
the divisors of a positive integer $n$ will be written as
$1=p_1<\ldots < p_{\tau(n)}=n$.

\begin{remark}
\label{rem:divisorDiagram}
Observe that for any box in the $l$th column of the divisor diagram,
if it belongs to a block of size $n$, then $\tau(p_l)=l$.
\end{remark}

\subsection{General combinatorial specifications for $\mathcal P$}

Let us formally give the two main global combinatorial specifications.

\begin{theorem}
\label{thm:ClassPExpo}
For every admissible triple $(i,j,k)$, define for every positive integer $n$ the arithmetic function
\begin{equation}
	\chi(n) := \begin{cases}
		n^2\cdot \tau_i(n) \quad & \textnormal{for } (i\geq 1,0,0) \smallskip \\
		n \sum\limits_{\ell=1}^{\tau(n)} p_\ell \cdot \tau_i(p_\ell) \cdot \tau_k(n/p_\ell) \quad & \textnormal{for } (i\geq 1,0,k \geq 1) \smallskip \\ 
		\sum\limits_{\ell=1}^{\tau(n)} p_\ell^2 \cdot \tau_i(p_\ell) \cdot \tau_j(n/p_\ell) \quad & \textnormal{for } (i\geq 1,j\geq 1, 0) \smallskip \\
		\sum\limits_{\ell=1}^{\tau(n)} \sum\limits_{m=1}^{\tau(n/p_{\ell})} p_\ell^2 q_m \cdot \tau_i(p_\ell)  \cdot \tau_k(q_m) \cdot \tau_j\big(n/(p_\ell q_m)\big) \quad & \textnormal{for } (i\geq 1,j\geq 1,k \geq 1) \\
		n\cdot \tau_k(n)  \quad & \textnormal{for } (0,0,k \geq 1) \smallskip \\  
		\sum\limits_{\ell=1}^{\tau(n)} p_\ell \cdot \tau_k(p_\ell) \cdot \tau_j(n/p_\ell) \quad & \textnormal{for } (0,j\geq 1,k \geq 1) \smallskip \\
		\tau_j(n) \quad & \textnormal{for } (0,j\geq 1,0) \smallskip \\
\end{cases}
\end{equation}
where 
$q_m | \frac{n}{p_\ell}$ for $m=1,\ldots,\tau(n/p_\ell)$ are the divisors of $n/p_\ell$.
Then $P(z)$ is the \emph{cyclic Euler transform} of $\chi(n)$, that is,
\begin{align}
\label{eq:generalizedEulerTransformation}
	P(z) = \prod_{n\geq 1} \exp\left( \sum\limits_{\ell=1}^\infty \frac{z^{n \ell }}{n\ell} \right)^{\chi(n)}.
\end{align}
%\comment{Since we made a note here about thedefinition of $q_{m}$ as the divisors of $n/p_{\ell}$, then we should probably make a note that $p_{\ell}$ are the divisors of $n$.}
\end{theorem}

\begin{proof}
We have
%\comment{Maybe we come back and reduce the number of lines in this argument, for the final journal paper version}
\begin{align}
	P(z) & = \prod \left(
		\frac{1}{1-z^{n_1\cdots n_{i}d_{1}\cdots d_{j}e_{1}\cdots e_{k}}}
	\right)^{n_1\cdots n_i/d_{1}\cdots d_{j}} \\
	& = \prod \exp\left(\log \left( \left(
		\frac{1}{1-z^{n_1\cdots n_{i}d_{1}\cdots d_{j}e_{1}\cdots e_{k}}}
	\right)^{n_1\cdots n_i/d_{1}\cdots d_{j}} \right)\right)\\
        & = \prod \exp\left(\frac{1}{d_1\cdots d_j}\log \left( \frac{1}{1-z^{n_1\cdots n_{i}d_{1}\cdots d_{j}e_{1}\cdots e_{k}}} \right)\right)^{n_1\cdots n_i}\\ 
        & = \prod \exp\left(\frac{1}{d_1\cdots d_j} \sum\limits_{\ell=1}^\infty \frac{z^{n_1\cdots n_{i}d_{1}\cdots d_{j}e_{1}\cdots e_{k} \ell }}{\ell} \right)^{n_1\cdots n_i}\\
        & = \prod \exp\left(n_1\cdots n_{i}e_{1}\cdots e_{k} \sum\limits_{\ell=1}^\infty \frac{z^{n_1\cdots n_{i}d_{1}\cdots d_{j}e_{1}\cdots e_{k} \ell }}{n_1\cdots n_{i}d_{1}\cdots d_{j}e_{1}\cdots e_{k}\ell} \right)^{n_1\cdots n_i}\\
        & = \prod \exp\left( \sum\limits_{\ell=1}^\infty \frac{z^{n_1\cdots n_{i}d_{1}\cdots d_{j}e_{1}\cdots e_{k} \ell }}{n_1\cdots n_{i}d_{1}\cdots d_{j}e_{1}\cdots e_{k}\ell} \right)^{(n_1\cdots n_i)^2\cdot e_{1}\cdots e_{k}}
\end{align}
and hence the result follows.
\end{proof}

\begin{corollary}
The exponential generating function in equation \eqref{eq:GeneralPartitionMulti} comes from the specification
\begin{align}
\label{eq:specGeneral}
	\mathcal P(\mathcal Z) = \prod\limits_{n=1}^\infty \seq_{\chi(n)}\left(\set\left(\bigcup_{\ell = 1}^{\infty}\cyc_{n\ell} (\mathcal Z)\right)\right).
\end{align}
\end{corollary}

Only when $j>0$ we are forced to consider cycles.
Thus the case $j=0$ deserves its own attention. In fact we will see now,
in the result that follows, that in this case $\mathcal P$ admits a specification
that translates into a representation of $P(z)$ with the form given in
Meinardus' Theorem (thus the later could be applied if all its hypothesis are satisfied);
the proof of this result and the next two that follow are now straightforward.

\begin{theorem}
\label{thm:ClassPOrdi}
For every positive integer $n$, let
\begin{equation}
	\psi(n) := \begin{cases}
		n \cdot  \tau_i(n) \quad & \textnormal{for } (i \geq 1, 0, 0) \smallskip \\ 
		\tau_k(n) \quad & \textnormal{for } (0, 0, k\geq 1) \smallskip \\ 
		\sum\limits_{\ell=1}^{\tau(n)} p_\ell \cdot \tau_i(p_\ell) \cdot \tau_k(n/p_\ell) \quad & \textnormal{for } (i\geq 1, 0, k\geq 1).
\end{cases}
\end{equation}
If $j=0$, then $P(z)$ is the Euler transform of $\psi(n)$, that is,
\begin{align}
\label{eq:EulerTransformation}
	P(z) & = \prod \left(
		\frac{1}{1-z^{n_1\cdots n_{i}e_{1}\cdots e_{k}}}
	\right)^{n_1\cdots n_i} 
	= \prod_{n\geq 1} \left(
		\frac{1}{1-z^{n}}
	\right)^{\psi(n)} .
\end{align}
\end{theorem}

\begin{corollary}
The exponential generating function in equation \eqref{eq:GeneralPartitionPower} comes from the specification
\begin{align}
\label{eq:specj0}
	\mathcal P(\mathcal Z) = \prod\limits_{n=1}^\infty \seq_{\psi(n)}\big(\seq(\mathcal Z^n)\big).
\end{align}
\end{corollary}

\begin{proposition}
\label{prop:isoClasses}
Given an admissible triple $(i, 0, k)$, the combinatorial classes defined by equations \eqref{eq:specGeneral} and \eqref{eq:specj0} are isomorphic.
\end{proposition}

\subsection{Examples}

So, let us now see several examples
to better understand the labeled structures defined by
equations \eqref{eq:specGeneral} and \eqref{eq:specj0}.
Let us start with the cases that force
an exponential context, i.e. when $j\geq 1$.
Actually, in virtue of Theorem \ref{thm:ClassPExpo},
analyzing the first few cases suffices to understand
the whole class described in equation \eqref{eq:specGeneral}.

%\subsection{Colored permutations by $j$-fold divisor function $(0,j \geq 1,0)$ }
%\label{subsec:00}

\subsubsection{Colored permutations by number of divisors: $(0,1,0)$ }
\label{subsubsec:010}

Consider the case $(i,j,k) = (0,1,0)$.
This is the case of Theorem \ref{thm:KotesovecP}.
In this case, the specification in equation \eqref{eq:specGeneral} becomes
\begin{align}
\label{eq:specLabeledCycles}
	\mathcal{P} (\mathcal Z) 
	= \prod_{n=1}^{\infty}\set\Big(\bigcup_{\ell = 1}^\infty
	\cyc_{n\ell}(\mathcal{Z})\Big)
\end{align}
(compare with \eqref{eq:permutationSpec}).
It represents the class of permutations, but now they carry a label on each cycle
\emph{which is a divisor of that cycle's length}.
For instance,
\begin{align}
\label{eq:exLabelPartition}
	\overbrace{(3\thinspace 19\thinspace 30\thinspace 21\thinspace 14\thinspace
	11\thinspace 9\thinspace 15\thinspace 29\thinspace 6\thinspace 8\thinspace 32)}^{3|12}
	\overbrace{(2\thinspace 10\thinspace 16\thinspace 13\thinspace 34\thinspace 20)}^{2|6}
	\overbrace{(24\thinspace 26\thinspace 33\thinspace 25\thinspace 28)}^{1|5}
	\overbrace{(4\thinspace 18\thinspace 22\thinspace 23)}^{4|4}	
	\overbrace{(1\thinspace 5\thinspace 27)}^{3|3}
	\overbrace{(7\thinspace 12\thinspace 31)}^{1|3}
	\overbrace{(17)}^{1|1}.
\end{align}
%and the above is an example of an object of size~34 from this class of labelled permutations.
Any such set of labels showing an example divisor of each cycle length
is suitable.  For instance, another element of size~34 is
\begin{align}
\label{eq:exLabelPartition2}
	\overbrace{(3\thinspace 19\thinspace 30\thinspace 21\thinspace 14\thinspace
	11\thinspace 9\thinspace 15\thinspace 29\thinspace 6\thinspace 8\thinspace 32)}^{12|12}
	\overbrace{(2\thinspace 10\thinspace 16\thinspace 13\thinspace 34\thinspace 20)}^{3|6}
	\overbrace{(24\thinspace 26\thinspace 33\thinspace 25\thinspace 28)}^{5|5}	
	\overbrace{(4\thinspace 18\thinspace 22\thinspace 23)}^{2|4}
	\overbrace{(1\thinspace 5\thinspace 27)}^{1|3}
	\overbrace{(7\thinspace 12\thinspace 31)}^{1|3}
	\overbrace{(17)}^{1|1}
\end{align}
We emphasize that \emph{any} label of the form $p|l$ above a cycle
of length $l$ is valid. In other words, an element of
$\mathcal{P}$ is completely determined by the structure of the cycles
in that permutation, and the choice of one divisor (indeed, any
divisor $p$) of each cycle length.
The choice of divisors is equivalent to coloring each row of length $n$ of the corresponding Young diagram
with $\tau(n)$ colors, e.g. see Figure \ref{fig:PartiLabel}.
So in this case
the structure is the class of \emph{colored permutations with divisor function} $\tau(n)$,
that is, colored permutations \emph{by number of divisors}.
Its first elements are shown in Figure \ref{fig:elementsPQ}.\begin{figure}[h] %  figure placement: here, top, bottom, or page
   \centering
   \includegraphics[width=6.2in]{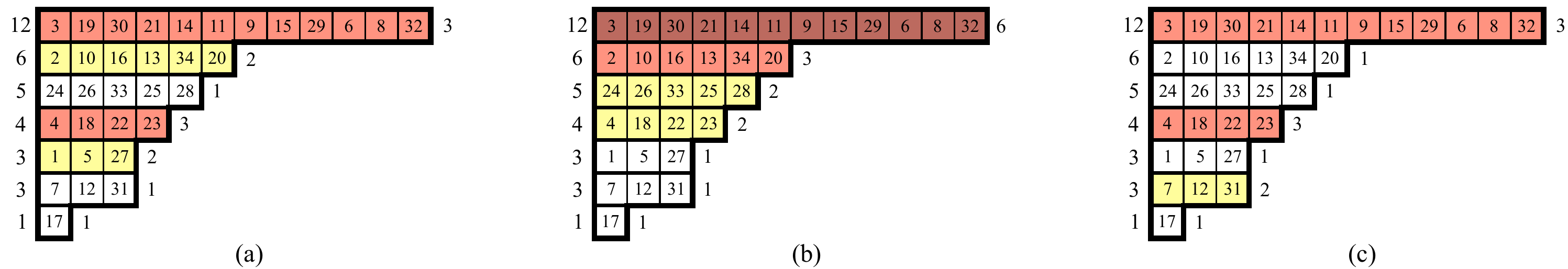} 
   \caption{Colored permutations by number of divisors (the numbering on the right
   of each row corresponds to the color number).
   (a) Coloring according to equation \eqref{eq:exLabelPartition} 
   (b) Coloring according to equation \eqref{eq:exLabelPartition2}.
   (c) Another legal coloring.}
   \label{fig:PartiLabel}
\end{figure}
\begin{figure}[h] %  figure placement: here, top, bottom, or page
   \centering
   \includegraphics[width=6.0in]{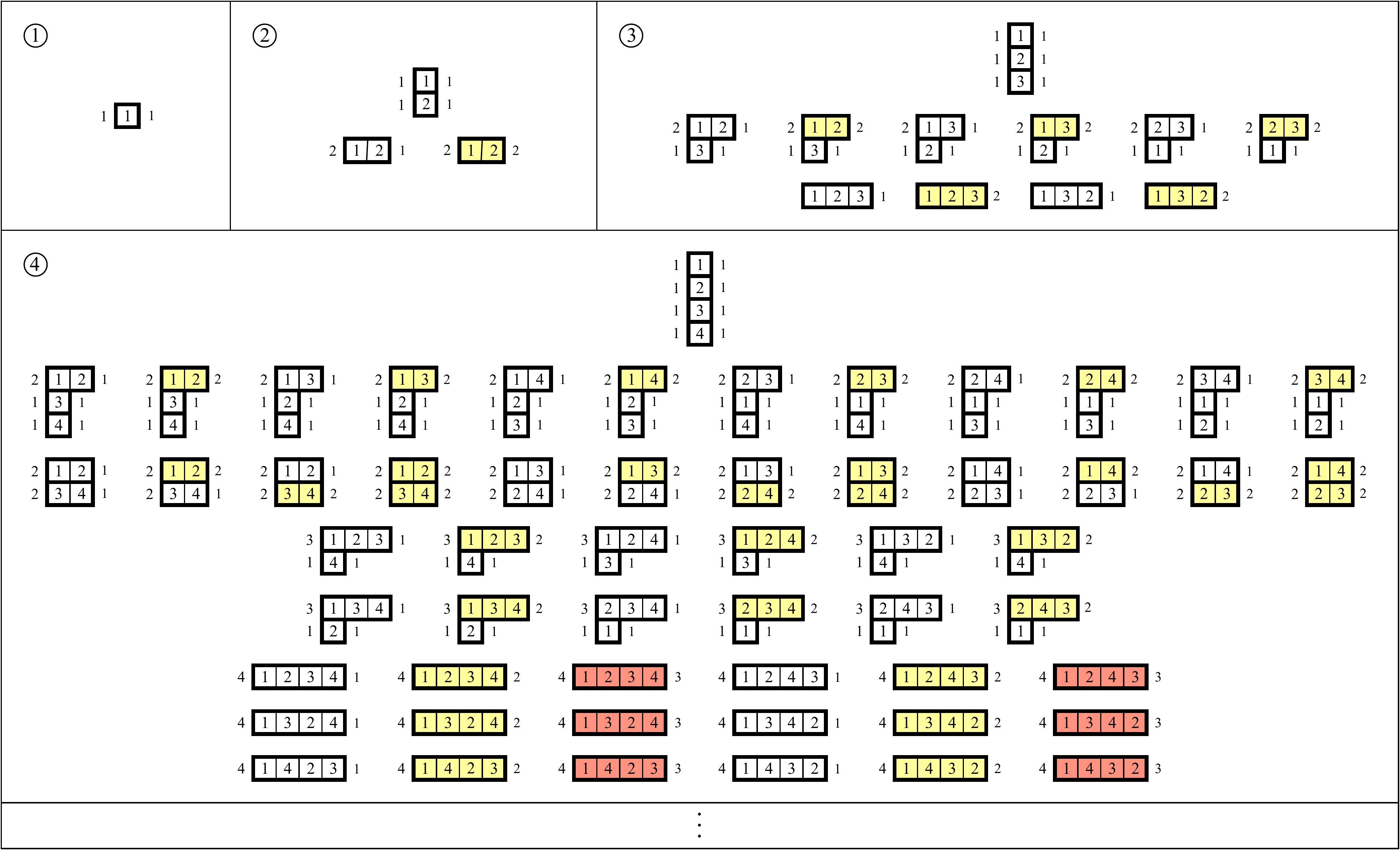} 
   \caption{The first elements of colored permutations by number of divisors.
   %The number on the left of the Young diagram form the corresponding integer partition.
   %The numbers on the right of the Young diagram correspond to the coloring.
   }
   \label{fig:elementsPQ}
\end{figure}

\subsubsection{Colored permutations by number of divisors
with rooted colorings of divisor diagrams by number of divisors: $(0,2,0)$ }
\label{subsubsec:020}

Consider the case $(i,j,k) = (0,2,0)$.
In this case we obtain the specification
\begin{align}
\label{eq:specLabeledCyclesDivisor}
	\mathcal{P} (\mathcal Z) =
	\prod_{n=1}^{\infty}\seq_{\tau(n)}\left(\set\left( \bigcup\limits_{\ell=1}^\infty \cyc_{n\ell}(\mathcal Z)\right)\right).
\end{align}
\begin{figure}[h] %  figure placement: here, top, bottom, or page
   \centering
   \includegraphics[width=6.5in]{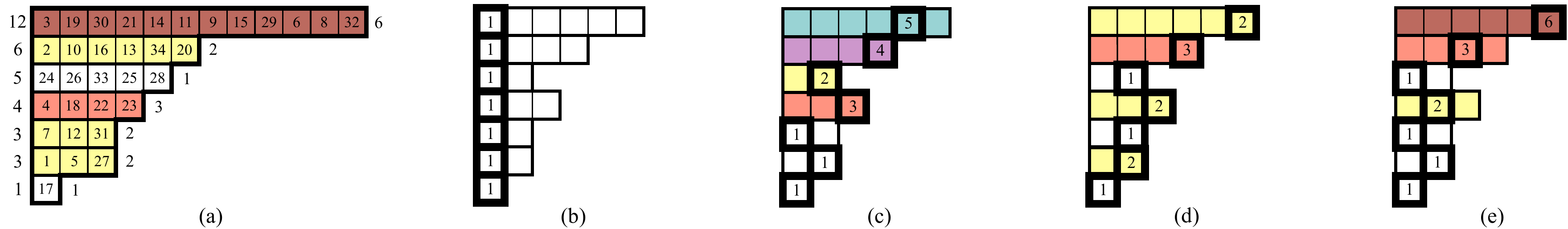} 
   \caption{A colored permutation by number of divisors
   together with four instances of rooted colorings of its divisor diagram by number of divisors.
   (a) A colored partition in a Young diagram.
   (b), (c), (d) Three different rooted colorings of the divisor diagram associated to the Young diagram in (a).}
   \label{fig:YoungDivisor}
\end{figure}

\noindent
Keeping in mind the case $(0,1,0)$ in Section~\ref{subsubsec:010},
we see that again we get colored permutations by number of divisors,
but now, in addition, each block in the Young diagram carries an additional
label, namely, for each block of length $n$, if we list the divisors of $n$ as
$1=p_1< \ldots < p_{\tau(n)}=n$ and then pick one divisor $p_l$, then the label
is a color chosen from a set of colors of size of $\tau(p_l)=l$
% we can think that each $n$-block
%now carries an additional label which is one of the divisors $p_l$ for some $l = 1 , \ldots , \tau(n)$,
%and then the block is colored with a set of colors of size $\tau(p_l) = l$
(see Remark \ref{rem:divisorDiagram}), and this can be
pictured as \emph{colorings of rooted divisor diagrams by size}.
So in this case, if the $l$th box of a block is the root, then the
size of the set of colors is $l$. See Figure~\ref{fig:YoungDivisor} for instance,
it shows a colored permutation and four instances of rooted colorings
of its divisor diagram by number of divisors. If the first box is chosen as
the root for all blocks, then there is only one coloring,
the \emph{monochromatic}, and this is the case shown in item (b)
in this figure. Items (c), (d) and (e) show other possible rooted colorings of the divisor diagram.
%(thus, in virtue of Theorem \ref{thm:ClassPExpo}, when the admissible triple is $(0,1,0)$, i.e. the case of
%previous subsection~\ref{subsubsec:010}, we also have colored permutations by number of divisors
%\emph{together} with \emph{monochomatic divisor diagrams},
%i.e. with divisor diagram function $\chi(n)=\tau_1(n)=1$; see what follows).

\subsubsection{Colored permutations by number of divisors with rooted colorings of divisor diagram with divisor function $\chi(n)$: $(i,j,k)$ }
\label{subsubsec:ijk}

With the two previous examples on mind, now we can understand the general case
in equation \eqref{eq:specGeneral} because it only differs
from the specification in equation \eqref{eq:specLabeledCyclesDivisor} 
by the divisor function $\chi(n)$ instead of the divisor function $\tau(n)$ in the sequence operator.
Thus, the structure is essentially the same, that is, the class consists again of colored
permutations by number of divisors, and the difference is that the coloring
of the rooted divisor diagram is now determined by the divisor function $\chi(n)$.
To be precise, if the $l$th box in a block of the divisor diagram
that is associated to an $n$-block in the Young diagram has been chosen as the root,
and if the divisors of $n$ are written as $1=p_1<\ldots <p_{\tau(n)}=n$,
then the block of the divisor diagram is colored with a set of colors of size $\chi(p_l)$.
Note that in the general case, if two distinct boxes in the $l$th column of the
divisor diagram are chosen as roots of their corresponding blocks,
then the blocks may be colored with sets of colors of distinct cardinalities.

\medskip
Now it is turn to analyze the specification given in equation \eqref{eq:specj0}.

\subsubsection{Ordered colorings of Young tableaux by size: $(1,0,0)$ }
\label{subsubsec:100}

Now consider the triple $(i,j,k) = (1,0,0)$.
In this case, $P(z)$ is known
to be the ordinary generating function of the \emph{plane partitions}, studied by Wright
(see \cite{Wright31}). In a labelled universe, the class consists of well colorings
of plane partitions and is combinatorially isomorphic to the (isomorphic) classes that result from the specifications
given in equations \eqref{eq:specGeneral} and \eqref{eq:specj0}.
The later becomes
\begin{align}
\label{eq:plane}
	\mathcal P(\mathcal Z) = \prod_{n=1}^{\infty}\seq_n\left({\seq(\mathcal Z^n)}\right)
\end{align}
and it translates into
\begin{align}
	P(z) = \prod_{n=1}^{\infty} \frac{1}{(1-z^n)^n}.
\end{align}
In this case, the labelled class $\mathcal P$, as specified in equation \eqref{eq:plane},
consists of \emph{colorings of Young tableaux by size}, i.e. the blocks of size $n$
are colored with $n$ distinct colors. See Figure \ref{fig:YoungColored} for examples.
Note the difference with the colored permutations by number of divisors from Section~\ref{subsubsec:010}
in that here the Young tableaux are colored by size and they are not modular nor symmetric.
Furthermore, in this case we also have $\chi(n)=n^2\tau_{1}(n)=n^2$, thus,
according to Proposition \ref{prop:isoClasses}, the class of colored Young tableaux by size
is isomorphic to the class of colored permutations by number of
divisors with rooted colorings of divisor diagrams by \emph{size squared}.
\begin{figure}[h] %  figure placement: here, top, bottom, or page
   \centering
   \includegraphics[width=6.5in]{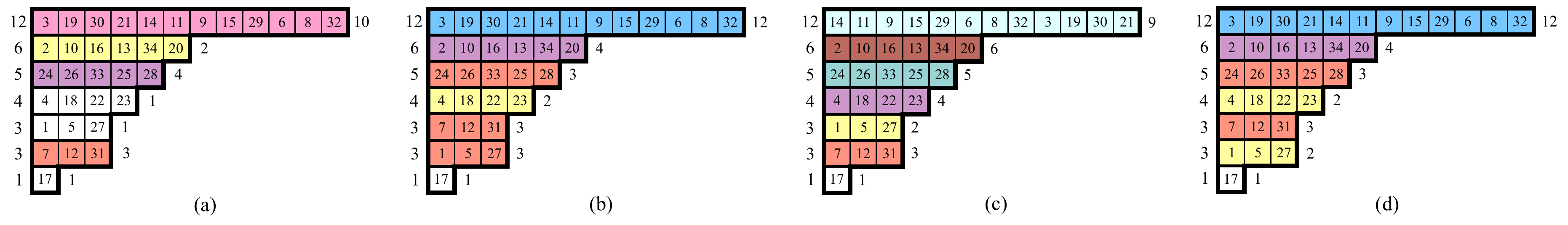} 
   \caption{Colored Young tableaux by size. Each row of length $n$ is colored with a set of colors of size $n$.
   (a) A colored Young tableau by size function.
   (b) Another colored Young tableau, it differs from (a) not only from the coloring, but also 
   the labelling of the rows of size three have been permuted, i.e. the Young tableaux are not symmetric.
   (c) Another colored Young tableau, it differs from (a) not only from the coloring, but also a
   row has been cyclically shifted, i.e. the Young tableaux are not modular.   
   (d) Another coloring of another Young tableau.
   }
   \label{fig:YoungColored}
\end{figure}

\subsubsection{Colored Young tableaux by divisor function $\psi(n)$: $(i,0,k)$}
\label{subsubsec:200}

The specification in equation \eqref{eq:specj0} for the general case of triples $(i,0,k)$
differs from the previos example in that $\psi(n)$ now replaces the size function $n$.
In other words, we get again colored Young tableaux like in Section~\ref{subsubsec:100}, the only difference is that now the
size of the set of colors used to color the set of $n$-blocks is the divisor function $\psi(n)$
(indeed, when $i=1$ and $k=0$, $\psi(n)=n$, i.e. $\psi(n)$ is \emph{the size}).
Again we recall that this class is isomorphic to the class of colored permutations by number of divisors with with rooted colorings of divisor diagrams with divisor function $\chi(n)$.

\subsection{The class $\mathcal Q$}

With the two global combinatorial perspectives for $\mathcal P$ given by the specifications in equation \eqref{eq:specGeneral}
and \eqref{eq:specj0},
we can give a general description of the corresponding class $\mathcal Q$
with markers and bivariate (exponential) generating functions as follows.
First consider a weighted version of equation \eqref{eq:specGeneral},
namely, the specification
\begin{align}
\label{eq:specLabeledCycles}
	\widehat{\mathcal{P}}(\mathcal Z , \mathcal V) = \prod\limits_{n=1}^\infty
	\seq_{\chi(n)}\left(\set\left(\bigcup_{\ell = 1}^{\infty} \mathcal V^{\ell + 1}\cyc_{n\ell} (\mathcal Z)\right)\right).
\end{align}
The corresponding bivariate (exponential) generating function $\widehat{P}(z,v)$ is therefore
\begin{align}
	\widehat{P}(z,v) =
	\prod_{n=1}^{\infty}\exp\left(\sum_{\ell = 1}^\infty
	v^{\ell+1}\frac{z^{n\ell}}{n\ell}\right)^{\chi(n)}.
\end{align}
In the case $v=1$, we have already seen (in equation~\eqref{eq:generalizedEulerTransformation}) that $\widehat{P}(z,1)=P(z)$.
Now set the weight $v=-1$ and follow the proof of Theorem \ref{thm:ClassPExpo}
backwards to get
\begin{align}
	\widehat{P}(z,-1)
	& = \prod_{n=1}^{\infty}\exp\left(\sum_{\ell=1}^{\infty} (-1)^{\ell+1}\frac{z^{n\ell}}{n\ell}\right)^{\chi(n)} \\
	& = \prod \exp\left( -\sum\limits_{\ell=1}^\infty \frac{(-z^{n_1\cdots n_{i}d_{1}\cdots d_{j}e_{1}\cdots e_{k}})^ \ell }{n_1\cdots n_{i}d_{1}\cdots d_{j}e_{1}\cdots e_{k}\ell} \right)^{(n_1\cdots n_i)^2\cdot e_{1}\cdots e_{k}} \\
	& = \prod \exp\left( -\sum\limits_{\ell=1}^\infty \frac{(-z^{n_1\cdots n_{i}d_{1}\cdots d_{j}e_{1}\cdots e_{k}})^ \ell }{\ell} \right)^{n_1\cdots n_i/ d_{1}\cdots d_{j}} \\
	& = \prod \exp\left( - \log \frac{1}{1+z^{n_1\cdots n_{i}d_{1}\cdots d_{j}e_{1}\cdots e_{k}}} \right)^{n_1\cdots n_i / d_{1}\cdots d_{j}} \\
	& = Q(z).
\end{align}
More generally, however, if we fix the value of the weights $v$, we
get 
\begin{align}
P_{n}^{(v)} = [z^{n}]\widehat{P}(z,v)
= [z^{n}]\prod_{n=1}^{\infty}\exp\left(\sum_{\ell = 1}^\infty
	v^{\ell+1}\frac{z^{n\ell}}{n\ell}\right)^{\chi(n)}
= [z^{n}]\prod_{\delta=1}^{\infty}\exp\Bigg(\sum_{\ell \geq 1\atop{\rm with\!\ } \delta|\ell}^\infty
	v^{\frac{\ell}{\delta}+1}\frac{z^{\ell}}{\ell}\Bigg)^{\chi(\delta)}.
\end{align}
%BACK TO OLD STUFF AGAIN
%We need to do the following in general
%$$
%	P_{n}^{(v)} = [z^{n}]\widehat{P}(z,v) =
%	\sum_{p\in\mathcal{P}}\prod_{ {\rm labels\!\ }d|\ell\atop{\rm from\!\ }p }v^{\frac{\ell}{d}+1}
%$$
%or equivalently
%$$
%	P_{n}^{(v)} = [z^{n}]\widehat{P}(z,v) = \sum_{p\in\mathcal{P}}v^{\sum(\frac{\ell}{d}+1)}
%$$
%where the summation in the power of $v$ is over all 
%labels $d|\ell$ from $p$.
%This immediately implies that t
As seen above, the coefficients $Q_{n} = [z^{n}]Q(z)$
can be calculated as a sum of weighted terms (with weights $+1$ and $-1$)
of coefficients from $P(z)$.
%as follows:
%$$
%	Q_{n} = [z^{n}]Q(z) = \sum_{p\in\mathcal{P}}
%	\prod_{ {\rm labels\!\ }d|\ell\atop{\rm from\!\ }p }(-1)^{\frac{\ell}{d}+1}
%$$
%or equivalently
%$$
%	Q_{n} = [z^{n}]Q(z) = \sum_{p\in\mathcal{P}}(-1)^{\sum(\frac{\ell}{d}+1)}
%$$
%where the summation in the power of $v$ is (as before) over all 
%labels $d|\ell$ from $p$.
Since $v=1$ yields $\widehat{P}(z,1) = P(z)$ 
and $v=-1$ yields $\widehat{P}(z,-1) = Q(z)$, it would be interesting
to explore other exponential generating functions of the form
$\widehat{P}(z,v)$ for other roots of unity $v=e^{i\theta}$ on the unit circle.
We leave such explorations as open problems.

\section{Asymptotic analysis}
\label{sec:asymptotic}

The asyptotic analysis is carried out in a similar way as in \cite{FlajoletSedewick09}.
Let us start with the substitution $z\triangleq e^{-t}$.
Let $L_{P}(t) = \log\big(P(e^{-t})\big)$ and
$L_{Q}(t) = \log\big(Q(e^{-t})\big)$.
Also, let $L(t)$ denote either $L_P(t)$ or $L_Q(t)$.

\subsection{Mellin transforms}

As a general reference to Mellin transforms, see \cite[Ch.~9]{Szpankowski01}.

\begin{proposition}[Mellin transform for $\mathcal P$ and $\mathcal Q$]
\label{prop:MellinMultiSet}
	Let $L_{P}^{*}(s) := {\mathcal M}[L_{P}(t); s]$ and
	$L_{Q}^{*}(s) := {\mathcal M}[L_{Q}(t); s]$
	be the Mellin transforms of $L_{P}(t)$ and $L_{Q}(t)$, respectively.
	These Mellin transforms have the succinct forms
	\begin{equation}
	\label{eq:MellinMultiSet}
		L_{P}^{*}(s) = \zeta(s-1)^i\zeta(s+1)^{j+1}\zeta(s)^{k}\Gamma(s)
	\end{equation}
and
	\begin{equation}
	\label{eq:MellinSet}
		L_{Q}^{*}(s) = (1-2^{-s})\zeta(s-1)^i\zeta(s+1)^{j+1}\zeta(s)^{k}\Gamma(s).
	\end{equation}
\end{proposition}

\begin{proof}%[Proof of Proposition \ref{prop:MellinMultiSet}]
We have 
$L_{P}(t) = \log\big(P(e^{-t})\big)
	 = \log \prod (1-e^{-tn_1\cdots n_id_1\cdots d_j e_1 \cdots
           e_k})^{-n_1\cdots n_i/d_1\cdots d_j}$,
by the definition of $L_P$ and $P$.  The log of the product equals the
sum of the logs, and it follows that 
$L_{P}(t) = \sum\frac{n_1\cdots n_i}{d_{1}\cdots d_{j}}
		\log(1/(1-e^{-tn_1\cdots n_id_1\cdots d_j e_1 \cdots
                  e_k}))$.  Then we can expand the log as a series,
                and exchange the order of the summations, to obtain
$L_{P}(t) = \sum_{j=1}^{\infty}\frac{1}{j}\sum
		\frac{n_1\cdots n_i}{d_{1}\cdots d_{j}}e^{-tn_1\cdots n_id_1\cdots d_j e_1 \cdots e_kj}$.
By linearity of the Mellin transform, we have 
$L_{P}^{*}(s) = \sum_{j=1}^{\infty}\frac{1}{j}\sum
		\frac{n_1\cdots n_i}{d_{1}\cdots d_{j}}{\mathcal M}\left[e^{-tn_1\cdots n_id_1\cdots d_j e_1 \cdots e_kj} ; s\right]$.
Using scaling for the Mellin transform~\cite[equation~(9.5)]{Szpankowski01}, we have
${\mathcal M}\left[e^{-tn_1\cdots n_id_1\cdots d_j e_1 \cdots e_kj} ; s\right]
= (n_1\cdots n_id_1\cdots d_j e_1 \cdots e_kj)^{-s}{\mathcal
  M}\left[e^{-t} ; s\right]$.  It is well known that ${\mathcal
  M}\left[e^{-t} ; s\right] = \Gamma(s)$.  
So it follows that 
$L_{P}^{*}(s)
	= \sum_{j=1}^{\infty}\frac{1}{j}\sum
		\frac{n_1\cdots n_i}{d_{1}\cdots d_{j}}
                \cdot(n_1\cdots n_id_1\cdots d_j e_1 \cdots e_k
                j)^{-s}\Gamma(s)$, which simplifies to 
$$L_{P}^{*}(s) = \zeta(s-1)^i\zeta(s+1)^{j+1}\zeta(s)^{k}\Gamma(s).$$
The analysis for $L_{Q}(t)$ is similar.  We have
$L_{Q}(t)  
	  = \log \prod (1+e^{-tn_1\cdots n_id_1\cdots d_j e_1 \cdots e_k})^{n_1\cdots n_i/d_{1}\cdots d_{j}}$.
Then we expand as before, and we get 
$L_{Q}(t) = \sum_{j=1}^{\infty}\frac{(-1)^{j+1}}{j}\sum
		\frac{n_1\cdots n_i}{d_{1}\cdots d_{j}}e^{-tn_1\cdots n_id_1\cdots d_j e_1 \cdots e_kj}$.
Again using linearity and scaling, 
it follows that 
$L_{Q}^{*}(s) = \sum_{j=1}^{\infty}\frac{(-1)^{j+1}}{j}\sum
		\frac{n_1\cdots n_i}{d_{1}\cdots d_{j}}
                \cdot(n_1\cdots n_id_1\cdots d_j e_1 \cdots e_k
                j)^{-s}\Gamma(s)$, which simplifies to 
$L_{Q}^{*}(s) = \eta(s+1)\zeta(s-1)^i\zeta(s+1)^{j}\zeta(s)^{k}\Gamma(s)$,
where $\eta(s) = \sum_{n=1}^{\infty}\frac{(-1)^{n+1}}{n^s}$ is the
Dirichlet eta function.
For symmetry with the equation for $L_{P}^{*}(s)$, we transform this to:
$$
L_{Q}^{*}(s)
	= (1-2^{-s})\zeta(s-1)^i\zeta(s+1)^{j+1}\zeta(s)^{k}\Gamma(s).
$$
\end{proof}

Henceforth, let $L^*(s)$ denote either $L^*_P(s)$ or $L_Q^*(s)$.

\subsection{Singularity analysis}
We make a handful of singularity analysis observations.
%\comment{Question:  Do we want to make this into a Note/Remark or a Proposition?  It was previously a bullet list! ANSWER: YES}

\begin{proposition}
\label{prop:singular}
The following hold:
\begin{enumerate}
\item
\label{prop:singular:2}
If $i\geq 1$, then $L^*(s)$ has a singularity at $s=2$ generated by
$\zeta(s-1)^{i}$.  It is a pole of order $i\geq 1$, because the remaining terms in equations \eqref{eq:MellinMultiSet} and \eqref{eq:MellinSet} are nonzero at $s=2$.
\item
\label{prop:singular:1}
If $k \geq 1$, then $L^*(s)$ has a singularity at $s=1$ generated by
$\zeta(s)^k$.  It is a pole of order $k\geq 1$ because the remaining terms terms in equations \eqref{eq:MellinMultiSet} and \eqref{eq:MellinSet} are nonzero at $s=1$.
\item
\label{prop:singular:0}
Both $L_P^*(s)$ and $L_Q^*(s)$ have a singularity at $s=0$ generated
by $\zeta(s+1)^{j+1}\Gamma(s)$.  The singularity of $L_P^*(s)$ at
$s=0$ is a pole of order $j+2 \geq 2$ because the remaining terms in
\eqref{eq:MellinMultiSet} are nonzero at $s=0$.
In contrast, since $1-2^{-s}$ has a simple zero at $s=0$, 
then the singularity of $L_Q^*(s)$ at
$s=0$ is a pole of order $j+1 \geq 1$ because the remaining terms in
\eqref{eq:MellinSet} are nonzero at $s=0$.
\item
\label{prop:singular:i0}
If $i=0$, then $L^*(s)$ has a singularity at $s=-1$ generated by
$\Gamma(s)$.  It is a simple pole because the remaining terms in
\eqref{eq:MellinMultiSet} and \eqref{eq:MellinSet} are nonzero at
$s=-1$.  On the other hand, if $i\geq 1$, then the simple pole at $s=-1$ generated by $\Gamma(s)$ is cancelled by the zero of order $i\geq 1$ in $\zeta(s-1)^i$. 
\item
\label{prop:singular:k0}
If $k=0$, then $L^*(s)$ has a singularity at $s=-2n$ 
for each $n\geq 1$,
generated by $\Gamma(s)$.  It is a simple pole because the remaining
terms in \eqref{eq:MellinMultiSet} and \eqref{eq:MellinSet} are
nonzero at $s=-2n$. 
On the other hand, if $k\geq 1$, then the singularity at $s=-2n$ is cancelled by the zero of order $k\geq 1$ in $\zeta(s)^k$.
\item
\label{prop:singular:even}
For each $n\geq 1$, the simple pole at $s = -(2n+1)$ generated by $\Gamma(s)$ is cancelled by the zero of order $j+1\geq 1$ in $\zeta(s+1)^{j+1}$.
\end{enumerate}
\end{proposition}

\subsection{Residue analysis}
\begin{proposition}
\label{prop:poly:a}
For all $i\geq 1$ and $j\geq 0$ and $k\geq 0$, there is a 
computable polynomial $a(t)=a_0+\ldots + a_{i-1}t^{i-1}$ such that
$${\operatorname{Res}}\big(L^*(s)t^{-s}\big)_{s=2} 
			= \frac{a(\log t)}{t^2}.$$
For $L_P^*(s)$ and $L_Q^*(s)$ respectively, the values of $a_{i-1}$
are 
\begin{equation}\label{definitionofAconstant}
\frac{(-1)^{i-1}\zeta(3)^{j+1}\pi^{2k}}{6^{k}(i-1)!}
\qquad\hbox{and}\qquad
\frac{(3)
  (-1)^{i-1}\zeta(3)^{j+1}\pi^{2k}}{(4)(6^{k})(i-1)!}.
\end{equation}
\end{proposition}
We use $a_{i-1}$ very often in the following discussion, so we define
$A := a_{i-1}$ (so that the dependence on the values of $i$ and $j$
and $k$ is implicit; we also treat the relationship to $P$ or $Q$ as implicit).
\begin{proof}
We first note that 
${\operatorname{Res}}\big(\zeta(s+1)^i t^{-s}\big)_{s=0}$ is a
polynomial in $\log{t}$ of degree $i-1$, with leading coefficient 
$\frac{(-1)^{i-1}}{(i-1)!}$.  So we have
$${\operatorname{Res}}\big(\zeta(s+1)^i t^{-s}\big)_{s=0} =
\frac{(-1)^{i-1}}{(i-1)!}(\log{t})^{i-1} + o((\log{t})^{i-1}).$$
At $s=0$, we note that 
$\zeta(s+3)^{j+1}$, $\zeta(s+2)^{k}$,
and
$\Gamma(s+2)$ are all smooth,
and also that 
$\zeta(2) = \pi^{2}\!/6$
and $\Gamma(2) = 1$.
It follows that
$${\operatorname{Res}}\big(
\zeta(s+1)^i \zeta(s+3)^{j+1} \zeta(s+2)^{k} \Gamma(s+2) t^{-s}\big)_{s=0} =
\frac{(-1)^{i-1}\zeta(3)^{j+1}\pi^{2k}}{6^{k}(i-1)!}(\log{t})^{i-1}
 + o((\log{t})^{i-1}).$$
Shifting $s$ by 2, the LHS becomes 
${\operatorname{Res}}\big(
\zeta(s-1)^i \zeta(s+1)^{j+1} \zeta(s)^{k} \Gamma(s)
t^{-(s-2)}\big)_{s=2}
= t^{2}{\operatorname{Res}}\big(
L_P^*(s)t^{-s}\big)_{s=2}$ so we obtain
$${\operatorname{Res}}\big(
L_P^*(s)t^{-s}\big)_{s=2} =
\frac{1}{t^{2}}\bigg(\frac{(-1)^{i-1}\zeta(3)^{j+1}\pi^{2k}}{6^{k}(i-1)!}(\log{t})^{i-1}
 + o((\log{t})^{i-1})\bigg).$$
At $s=2$, we also note that 
$1 - 2^{-s} = 3/4$, so we obtain
$${\operatorname{Res}}\big(
L_Q^*(s)t^{-s}\big)_{s=2} =
\frac{1}{t^{2}}\bigg(\frac{(3)(-1)^{i-1}\zeta(3)^{j+1}\pi^{2k}}{(4)(6^{k})(i-1)!}(\log{t})^{i-1}
 + o((\log{t})^{i-1})\bigg).$$

\end{proof}

\begin{proposition}
\label{prop:poly:b}
For all $i\geq 0$ and $j\geq 0$ and $k\geq 1$, there is a 
computable polynomial $b(t)=b_0+\ldots + b_{k-1}t^{k-1}$ such that
$${\operatorname{Res}}\big(L^*(s)t^{-s}\big)_{s=1} 
			= \frac{b(\log t)}{t}.$$
For $L_P^*(s)$ and $L_Q^*(s)$ respectively, the values of $b_{k-1}$
are 
\begin{equation}\label{definitionofBconstant}
\frac{(-1)^{i+k-1}\pi^{2(j+1)}}{6^{j+1}2^{i}(k-1)!}
\qquad\hbox{and}\qquad
\frac{(-1)^{i+k-1}\pi^{2(j+1)}}{6^{j+1}2^{i+1}(k-1)!}.
\end{equation}
\end{proposition}
We define $B := b_{k-1}$ (the dependence on $i,j,k$ and on $P$ or $Q$ is implicit).
\begin{proof}
Using exactly the same observation as from the previous proposition,
we first note that 
${\operatorname{Res}}\big(\zeta(s+1)^k t^{-s}\big)_{s=0}$ is a
polynomial in $\log{t}$ of degree $k-1$, with leading coefficient 
$\frac{(-1)^{k-1}}{(k-1)!}$.  So we have
$${\operatorname{Res}}\big(
 \zeta(s+1)^{k} t^{-s}\big)_{s=0} =
\frac{(-1)^{k-1}}{(k-1)!}(\log{t})^{k-1}
 + o((\log{t})^{k-1}).$$
At $s=0$, we note that 
$\zeta(s)^{i}$, $\zeta(s+2)^{j+1}$, 
and
$\Gamma(s+1)$ are all smooth,
and also that 
$\zeta(0) = -1/2$
and $\zeta(2) = \pi^{2}\!/6$
and $\Gamma(1) = 1$.
It follows that
$${\operatorname{Res}}\big(
\zeta(s)^i \zeta(s+2)^{j+1} \zeta(s+1)^{k} \Gamma(s+1) t^{-s}\big)_{s=0} =
\frac{(-1)^{i+k-1}\pi^{2(j+1)}}{6^{j+1}2^{i}(k-1)!}(\log{t})^{k-1}
 + o((\log{t})^{k-1}).$$
Shifting $s$ by 1, the LHS becomes 
${\operatorname{Res}}\big(
\zeta(s-1)^i \zeta(s+1)^{j+1} \zeta(s)^{k} \Gamma(s)
t^{-(s-1)}\big)_{s=1}
= t{\operatorname{Res}}\big(
L_P^*(s)t^{-s}\big)_{s=1}$ so we obtain
$${\operatorname{Res}}\big(
L_P^*(s)t^{-s}\big)_{s=1} =
\frac{1}{t}\bigg(\frac{(-1)^{i+k-1}\pi^{2(j+1)}}{6^{j+1}2^{i}(k-1)!}(\log{t})^{k-1}
 + o((\log{t})^{k-1})\bigg).$$
At $s=1$, we also note that 
$1 - 2^{-s} = 1/2$, so we obtain
$${\operatorname{Res}}\big(
L_Q^*(s)t^{-s}\big)_{s=1} =
\frac{1}{t}\bigg(\frac{(-1)^{i+k-1}\pi^{2(j+1)}}{6^{j+1}2^{i+1}(k-1)!}(\log{t})^{k-1}
 + o((\log{t})^{k-1})\bigg).$$

\end{proof}

\begin{proposition}
\label{prop:poly:c}
For all nonnegative $i$, $j$, $k$, there is a 
computable polynomial $c(t)=c_0+\ldots + c_{j+1}t^{j+1}$ such that
$${\operatorname{Res}}\big(L_P^*(s)t^{-s}\big)_{s=0} 
			= c(\log t)$$
and the value of $c_{j+1}$ is
\begin{equation}\label{definitionofCconstant}
\frac{(-1)^{i+j+k+1}}{2^{k}(j+1)!\!\ 12^{i}}.
\end{equation}
\end{proposition}
We define $C := c_{j+1}$ (the dependence on $i,j,k$ and on $P$ is implicit).
\begin{proof}
We first note that 
${\operatorname{Res}}\big(\zeta(s+1)^{j+1}\Gamma(s)t^{-s}\big)_{s=0}$ is a
polynomial in $\log{t}$ of degree $j+1$, with leading coefficient 
$\frac{(-1)^{j+1}}{(j+1)!}$.  So we have
$${\operatorname{Res}}\big(
 \zeta(s+1)^{j+1}\Gamma(s) t^{-s}\big)_{s=0} =
\frac{(-1)^{j+1}}{(j+1)!}(\log{t})^{j+1}
 + o((\log{t})^{j}).$$
At $s=0$, we note that 
$\zeta(s-1)^{i}$ and $\zeta(s)^{k}$, 
are both smooth,
and also that 
$\zeta(-1) = -1/12$
and $\zeta(0) = -1/2$.
It follows that
$${\operatorname{Res}}\big(
L_P^*(s)t^{-s}\big)_{s=0}
= {\operatorname{Res}}\big(
\zeta(s-1)^i \zeta(s+1)^{j+1} \zeta(s)^{k} \Gamma(s) t^{-s}\big)_{s=0} =
\frac{(-1)^{i+j+k+1}}{2^{k}(j+1)!\!\ 12^{i}}(\log{t})^{j+1}
 + o((\log{t})^{j}).$$
\end{proof}
We define $D := d_{j}$ (the dependence on $i,j,k$ and on $Q$ is implicit).

\begin{proposition}
\label{prop:poly:d}
For all nonnegative $i$, $j$, $k$, there is a 
computable polynomial $d(t)=d_0+\ldots + d_{j}t^{j}$ such that
$${\operatorname{Res}}\big(L_Q^*(s)t^{-s}\big)_{s=0} 
			= d(\log t)$$
and the value of $d_{j}$ is 
\begin{equation}\label{definitionofDconstant}
\ln(2)\frac{(-1)^{i+j+k}}{2^{k}j!\!\ 12^{i}} .
\end{equation}
\end{proposition}
\begin{proof}
We first note that 
${\operatorname{Res}}\big((1-2^{-s})\zeta(s+1)^{j+1}\Gamma(s)t^{-s}\big)_{s=0}$ is a
polynomial in $\log{t}$ of degree $j$, with leading coefficient 
$\ln(2)\frac{(-1)^{j}}{j!}$.  So we have
$${\operatorname{Res}}\big(
(1-2^{-s})\zeta(s+1)^{j+1} \Gamma(s)t^{-s}\big)_{s=0}\\
=
\ln(2)\frac{(-1)^{j}}{j!}(\log{t})^{j}
 + o((\log{t})^{j-1}).$$
At $s=0$, we note that 
$\zeta(s-1)^{i}$ and $\zeta(s)^{k}$, 
are both smooth,
and also that 
$\zeta(-1) = -1/12$
and $\zeta(0) = -1/2$.
It follows that
\begin{align*}
{\operatorname{Res}}\big(
L_Q^*(s)t^{-s}\big)_{s=0}
&= {\operatorname{Res}}\big(
(1-2^{-s})\zeta(s-1)^i \zeta(s+1)^{j+1} \zeta(s)^{k} \Gamma(s)t^{-s}\big)_{s=0}\\
&=
\ln(2)\frac{(-1)^{i+j+k}}{2^{k}j!\!\ 12^{i}}(\log{t})^{j}
 + o((\log{t})^{j-1}).
\end{align*}
\end{proof}

\begin{proposition}
\label{prop:poly:h}
Fix an admissible triple $(i,j,k)$. Then the following hold:
\begin{itemize}
\item
	If $i=0$, then ${\operatorname{Res}}\big(L^*(s)t^{-s}\big)_{s=-1}$
	is a linear monomial in $t$ (that depends on whether we are analyzing $\mathcal P$ or $\mathcal Q$).
\item
	Let $n \geq 1$. If $k=0$, then ${\operatorname{Res}}\big(L^*(s)t^{-s}\big)_{s=-2n}$ is a degree $2n$ monomial in $t$ (that depends on whether we are analyzing $\mathcal P$ or $\mathcal Q$).
\end{itemize}
\end{proposition}

When necessary, we will add a subscript in the polynomials $a(t)$ and $b(t)$
to indicate which class we are considering, that is, we will write $a_P (z)$,
$a_Q (z)$, $b_P (z)$ and $b_Q (z)$.
Only in \emph{four} exceptional cases we will be able to find
the asymptotic growth of the coefficients of $P(z)$ namely
when the admissible triple is $(1,0,0)$, $(0,1,0)$, $(0,0,1)$, and $(1,0,1)$, and similarly, only in \emph{five} exceptional cases we will be able to find the asymptotic growth of the coefficients of $Q(z)$, namely when the admissible triple is one of the previous four, and, in addition, when the admissible triple is $(0,2,0)$.
For all these cases, we will need all the coefficients of the corresponding polynomials.
Tables \ref{table:polynomialsP} and \ref{table:polynomialsQ} show the polynomials in Propositions \ref{prop:poly:a}, \ref{prop:poly:b}, \ref{prop:poly:c} and \ref{prop:poly:d}.
for all these cases.

\begin{table}[h!]
\centering
	\begin{tabular}{|c | c | c | c |}
		\hline
		$(i,j,k)$ & $a(z)$ & $b(z)$ & $c(z)$\\
		\hline \hline \hline
		$(1,0,0)$ & $\zeta(3)$ & 
		& $\zeta'(-1)+\frac{1}{12}z$ \\
		\hline
		$(1,0,1)$ & $\frac{\pi^2\zeta(3)}{6}$ 
		& $-\frac{\pi^2}{12}$ 
		& $-\frac{\zeta'(-1)}{2} + \frac{\log (2\pi)}{24} - \frac{1}{24}z$ \\
		\hline \hline
		$(0,0,1)$ & 
		& $\frac{\pi^2}{6}$ 
		& $-\frac{\log(2\pi)}{2} + \frac{1}{2}z$ \\
		\hline \hline
		$(0,1,0)$ & & 
		& $\frac{\pi^2}{12}-\frac{\gamma^2}{2} - 2\gamma_1 - \gamma z + \frac{1}{2}z^2$ \\
		\hline
	\end{tabular}
	\caption{Polynomials in Propositions \ref{prop:poly:a}, \ref{prop:poly:b} and \ref{prop:poly:c} associated to the first admissible triples for the class $\mathcal P$.}
	\label{table:polynomialsP}
\end{table}

\begin{table}[h!]
\centering
	\begin{tabular}{|c | c | c | c |}
		\hline
		$(i,j,k)$ & $a(z)$ & $b(z)$ & $d(z)$\\
		\hline \hline \hline
		$(1,0,0)$ & $\frac{3\zeta(3)}{4}$ & 
		& $-\frac{\log 2}{12}$ \\
		\hline
		$(1,0,1)$ & $\frac{\pi^2\zeta(3)}{8}$ 
		& $-\frac{\pi^2}{24}$ 
		& $\frac{\log 2}{24}$ \\
		\hline \hline
		$(0,0,1)$
		& & $\frac{\pi^2}{12}$ 
		& $-\frac{\log(2)}{2}$ \\
		\hline \hline
		$(0,1,0)$ & & 
		& $\gamma\log 2 - \frac{\log^22}{2} -(\log 2)z$ \\
		\hline
		$(0,2,0)$ & & 
		& $\frac{\gamma^2\log 2}{2} + \frac{\pi^2\log 2}{12}
- \gamma\log^22+\frac{\log^32}{6} - 3\gamma_1\log 2 + \big(\frac{\log^22}{2} - 2\gamma\log 2\big)z
		+ \frac{\log 2}{2}z^2$ \\
		\hline
	\end{tabular}
	\caption{Polynomials in Proposition \ref{prop:poly:a}, \ref{prop:poly:b} and \ref{prop:poly:d} associated to the first admissible triples for the class $\mathcal Q$.}
	\label{table:polynomialsQ}
\end{table}

\begin{proposition}
\label{prop:asymp}
	Let $g(z)$ be equal to either $c(z)$ or $d(z)$ depending on whether
	we are analyzing $\mathcal P$ or $\mathcal Q$, respectively, and
	also let $h(z)$ be the series formed by the sum of the residues at
	$s=-1$ and at $s=-2n$ with $n\geq 1$ (note that $h(z) = 0$ if $i>0$ and $k>0$,
	see \ref{prop:singular:i0} and \ref{prop:singular:k0} in Prop. \ref{prop:singular}).
	As $t\to 0$ the following hold:
	\begin{itemize}
	\item
		If $i\geq 1$, then $L(t) \sim \frac{a(\log t)}{t^2} + \frac{b(\log t)}{t} + g(\log t) + h(t)$.
	\item
		If $i=0$ and $k\geq 1$, then $L(t) \sim \frac{b(\log t)}{t} + g(\log t)+ h(t)$.
	\item
		If $i=k=0$, then $L(t) \sim g(\log t) + h(t)$.
	\end{itemize} 
\end{proposition}
Henceforth we will systematically divide our analysis according to the following three types of admissible triples:
$(i\geq 1, j, k)$, $(0,j,k\geq 1)$ and $(0,j\geq 1, 0)$.

\subsection{Saddle point equations}

Let $F(z)$ denote either $P(z)$ or $Q(z)$. Now using $t=-\log z$ as $z \to 1$ yields
%\begin{align}
%	\log P(z) =  \frac{a\big(\log (-\log z)\big)}{(-\log z)^2} + \frac{b\big(\log (-\log z)\big)}{-\log z} + g\big(\log (-\log z)\big)
%	+ O\big((-\log z)^2\big)
%\end{align}
%\comment{When $k=0$ we have no residues at negative integers, so we could remove $O\big((-\log z)^2\big)$, though I understand is irrelevant, right? Just to make sure I understand}
%and so
\begin{align}
\label{eq:FExpansion}
	F(z) \sim \exp\left(
	\frac{a\big(\log (-\log z)\big)}{(-\log z)^2} + \frac{b\big(\log (-\log z)\big)}{-\log z} + g\big(\log (-\log z) \big) + h(-\log z)\big)
	\right).
\end{align}
Let $f(z):=\log F(z)-n\log(z)$.
The saddle is located at radius ``$r$'' with $rf'(z) = r\frac{F'(r)}{F(r)}=n$. We have
the following elementary result that gives a general form:

\begin{proposition}
\label{prop:Pderivative1}
\begin{align}
\label{eq:Pderivative1}
	z\frac{F'(z)}{F(z)} & = \frac{1}{\log^3(1/z)} \cdot \left(2a\big( \log\log(1/z)\big)
	- a'\big( \log\log(1/z)\big) \right) \\
	\label{eq:Pderivative2}
	& \quad + \frac{1}{\log^2(1/z)} \cdot \left(b\big( \log\log(1/z)\big) - b'\big( \log\log(1/z)\big) \right)\\
	\label{eq:Pderivative3}
	& \quad - \frac{1}{\log(1/z)} \cdot g'\big( \log\log(1/z)\big) \\
	& \quad -h'\big( \log(1/z)\big).
\end{align}
\end{proposition}

In general, all the terms numbered as \eqref{eq:Pderivative1}--\eqref{eq:Pderivative3} need to be taken into account in order to determine the asymptotic growth of the coefficients of $F(z)$.
We can neglect the contribution of $h(z)$. Consequently, as we have already mentioned, only in a very few cases we can explicitly solve \emph{the saddle point equation} \begin{align}
\label{eq:spe}
	r\frac{F'(r)}{F(r)}=n
\end{align}
and thus it is only in these cases that an explicit expression for the asymptotic growth of the coefficients of $F(z)$ can be obtained. Let us precisely identify the cases we can resolve:

\begin{corollary} We have the following:
\label{cor:saddle1P}
	\begin{itemize}
	\item
		If the admissible triple is either $(1,0,0)$, $(1,0,1)$,
                $(0,0,1)$, or $(0,1,0)$, then $r\frac{P'(r)}{P(r)}=n$ becomes a polynomial in $-\log r$
		of degree at most 3 and thus it can be solved.
	\item
		If the admissible triple is either $(1,0,0)$,
                $(1,0,1)$, $(0,0,1)$, $(0,1,0)$, or $(0,2,0)$,
		then $r\frac{Q'(r)}{Q(r)}=n$ becomes a polynomial in $-\log r$
		of degree at most 3 and thus it can be solved.
	\end{itemize}
\end{corollary}

\begin{proof}
All follow from Propositions \ref{prop:poly:a}, \ref{prop:poly:b}, \ref{prop:poly:c}, \ref{prop:poly:d} and \ref{prop:poly:h}.
\end{proof}

Nevertheless, as we will see, it is possible to obtain the
asymptotic growth of the \emph{logarithm} of the coefficients for every admissible triple.
The next section contains the case studies where we solve the saddle point equation
and deduce the asymptotic growth of the coefficients for all the cases in Corollary \ref{cor:saddle1P}.
Then in section \S \ref{sec:general} we present the general result
that gives the asymptotic growth of the logarithm of the coefficients for every admissible triple.

\section{Case Studies: First-Order Asymptotic Approximation.}

In this section we deduce the asymptotic growth of the coefficients
in all the cases in Corollary \ref{cor:saddle1P} that are not known
%\comment{Perhaps we can leave cases $(0,0,1)$ and $(1,0,0)$ and  to the readers}.
(Table \ref{tab:asymptoticsOEIS} shows the current status of this estimates in OEIS).
We first recall the following:

\begin{remark}
\label{rem:Lambert}
	The solution to $x = be^{cx}$ is
	$x = - \frac{1}{c}\operatorname{W}(-bc)$,
	which is also the solution to $\log x = \log b + cx$.
\end{remark}

\begin{theorem}
\label{thm:010P}
	If the admissible triple is $(0,1,0)$, then
	\begin{align}
		[z^n]P(z)
		&\sim 	\frac{e^{\pi^{2}\!/12 - \gamma^{2}\!/2 - 2\gamma_{1} + \gamma}
\exp\left(
        \log^{2}( \operatorname{W}(e^\gamma n)/n  )/2\right)}
{\sqrt{2\pi}\big( \log(e^\gamma n)/n \big)^{\gamma}\sqrt{\log\big(\log(e^\gamma n)/n)\big)}}
	\end{align}
\end{theorem}
 
 \begin{proof}
From Table \ref{table:polynomialsP} we get
$c(z) =\frac{\pi^2}{12}-\frac{\gamma^2}{2} - 2\gamma_1 - \gamma z + \frac{1}{2}z^2$.
Then the saddle point equation becomes
$-\log(\log (1/z))/\log (1/z) + \gamma/\log (1/z)  = n,
%	-\frac{\log(\log (1/z))}{\log (1/z)} + \frac{\gamma}{\log (1/z)}  = n,
$
which is equivalent to $\log u = \log e^\gamma -nu$
with $u=\log (1/z)$. Thus, with Remark \ref{rem:Lambert},
we obtain the saddle point
$
	r = \exp\left(-\operatorname{W}(e^\gamma n)/n\right)
	%\sim \exp\left( - \frac{\log n}{n} \right)
	%= 1 - \frac{\log n}{n} + \ldots
$
and with it we get
\begin{align}
	\frac{1}{2\pi}e^{\log P(r) - n\log r} 
&\sim
	\frac{1}{2\pi}
\exp\left(
        \frac{\log^{2}( \operatorname{W}(e^\gamma n)/n  )}{2}\right)
( \operatorname{W}(e^\gamma n)/n )^{-\gamma}
e^{\pi^{2}\!/12 - \gamma^{2}\!/2 - 2\gamma_{1}}\exp\big( \operatorname{W}(e^\gamma n) \big)
\end{align}
and
\begin{align}
	\int_{-\pi}^\pi \exp\left( \frac{f''(r)}{2!}(re^{i\theta} - r)^2 + \ldots \right)d\theta
	& \sim
	\frac{\operatorname{W}(e^\gamma n)}{n}\sqrt{\frac{2\pi}{\log\big(\operatorname{W} (e^\gamma n)/n)\big)}}. %\cdot 
	%\frac{e^{\gamma\log2 - \frac{\log^22}{2}}}{\pi \log^{\log 2}2}\cdot n^{\log 2} \\
	%& \sim \sqrt{\frac{2}{\pi}} \cdot 
	%\frac{e^{\gamma\log2 - \frac{\log^22}{2}}}{ (-\log 2)^{1/2+\log 2}}\cdot n^{\log 2}
\end{align}
Putting together the last two estimates above, since $\operatorname{W}(x) = \log x -\log\log x + o(x)$, we get,
after simplifications, the desired result.
\end{proof}

\begin{theorem}
	If the admissible triple is $(0,1,0)$, then
	\begin{align}
		[z^n]Q(z) \sim 
	\frac{2^{\gamma-(\log2)/2+1/2}}{ \sqrt{\pi}(\log 2)^{\log 2-1/2}}\cdot n^{\log 2-1}.
	\end{align}
\end{theorem}

\begin{proof}
In this case, from Table \ref{table:polynomialsQ} we get
$d(z) = \gamma\log 2 - \frac{\log^22}{2}-(\log 2)z$.
Then the saddle point equation $(\log 2)/\log (1/z)  = n$ yields
the asymptotic saddle point
$r = \exp\left(-(\log 2)/n\right)
	%r & = \exp\left(-\frac{\log 2}{n}\right)
$
and with it we get
\begin{align}
	\frac{1}{2\pi}e^{\log Q(r) - n\log r} \sim
	\frac{e^{\gamma\log2 - \frac{\log^22}{2}}}{2\pi (\log 2)^{\log 2}}\cdot n^{\log 2} \cdot 2
\end{align}
and
\begin{align}
	\int_{-\pi}^\pi \exp\left( \frac{f''(r)}{2!}(re^{i\theta} - r)^2 + \ldots \right)d\theta
	& \sim \frac{\sqrt{2\pi\log 2}}{n} %\cdot 
	%\frac{e^{\gamma\log2 - \frac{\log^22}{2}}}{\pi \log^{\log 2}2}\cdot n^{\log 2} \\
	%& \sim \sqrt{\frac{2}{\pi}} \cdot 
	%\frac{e^{\gamma\log2 - \frac{\log^22}{2}}}{ (-\log 2)^{1/2+\log 2}}\cdot n^{\log 2}
\end{align}
Putting together the last two estimates above and simplifying yields the desired result.
\end{proof}

%\subsection{$(0,2,0)$ ??? (Gutkovskiy (2018)).}

\begin{theorem}
	If the admissible triple is $(0,2,0)$, then
	\begin{align}
		[z^n]Q(z) &\sim 
	\frac{e^{d_0+d_1\log \theta(n)+d_2\log^2\theta(n)-d_1}}{2\sqrt{\pi}
\sqrt{d_2\log\big(\frac{2d_2}{n}\log\big( \frac{ne^{-d_1/(2d_2)}}{2d_2} \big)\big)}
}
\Big(\frac{2d_2}{n}\log\Big( \frac{ne^{-d_1/(2d_2)}}{2d_2} \Big)\Big)^{1-2d_2}
	\end{align}
	where $\theta(n) = \frac{2d_2}{n}\operatorname{W}\big( \frac{ne^{-d_1/(2d_2)}}{2d_2} \big)$.
\end{theorem}

\begin{proof}
Technically, this theorem is identical to Theorem \ref{thm:010P}, so let us develop it implicitly.
The saddle point equation $-2d_2\log\log(1/z)/\log(1/z) -  d_1/\log(1/z)  = n$ is
equivalent to $\log u = -\frac{d_1}{2d_2}-\frac{n}{2d_2} u$ with $u=\log(1/z)$,
thus we get
the asymptotic saddle point
$r = \exp\big(-\frac{2d_2}{n}\operatorname{W}\big( e^{-d_1/(2d_2)}\cdot \frac{n}{2d_2} \big) \big)
	%r & = \exp\left(-\frac{\log 2}{n}\right)
$
and with it, letting $s = -\log r$, we get
\begin{align}
	\frac{1}{2\pi}e^{\log Q(r) - n\log r} \sim
	\frac{e^{d_0+d_1\log \theta(n)+d_2\log^2 \theta(n)}}{2\pi} \cdot e^{n\theta(n)} 
\end{align}
and
\begin{align}
	\int_{-\pi}^\pi \exp\left( \frac{f''(r)}{2!}(re^{i\theta} - r)^2 + \ldots \right)d\theta
	& \sim \sqrt{\frac{\pi }{d_2\log \theta(n)}} \cdot \theta(n)%\cdot 
	%\frac{e^{\gamma\log2 - \frac{\log^22}{2}}}{\pi \log^{\log 2}2}\cdot n^{\log 2} \\
	%& \sim \sqrt{\frac{2}{\pi}} \cdot 
	%\frac{e^{\gamma\log2 - \frac{\log^22}{2}}}{ (-\log 2)^{1/2+\log 2}}\cdot n^{\log 2}
\end{align}
Putting together the last two estimates above and simplifying yields the desired result.
\end{proof}

\section{First-Order Asymptotic Approximation of Logarithmic Growth.}
\label{sec:general}

Despite the fact that the saddle point equation above can be solved explicitly only in a few cases, to determine the asymptotic growth of the \emph{logarithm} of the coefficients, the following slower asymptotic saddle points suffice.
(Again we use the notation $F(z)$ to represent either $P(z)$ or $Q(z)$.)
\begin{proposition}
\label{prop:saddle}  
	The saddle point equation $rf'(r)=r\thinspace F^{\prime}(r) / F(r) - n = 0$ yields
	the following \emph{weak} asymptotic saddle points 
$r = \exp(-\exp(\alpha(n)))$ along the positive real axis:
		\begin{equation}
			\alpha(n) = \begin{cases}
  -\log\big(  \frac{n}{2A}  \big)^{1/3} & \hbox{if $i=1$}\\
 - \frac{i-1}{3} \operatorname{W}\big( - \frac{3}{i-1} \big(\frac{n}{2A }\big)^{1/(i-1)}    \big) & \hbox{if $i > 1$}\\
 -\log \big(  \frac{n}{B }  \big)^{1/2} & \hbox{if $i=0$ and $k=1$}\\
 - \frac{k-1}{2} \operatorname{W}\big( - \frac{2}{k-1}
 \big(\frac{n}{B}\big)^{1/(k-1)}    \big)  & \hbox{if $i=0$ and $k > 1$}\\
 -j\!\ \operatorname{W}\big( - \frac{1}{j} \big(\frac{n}{-(j+1)C }\big)^{1/j}    \big)  &
\hbox{if $i=0$ and $k=0$ and $j\geq 1$ and $F(z)=P(z)$}\\
\log( {-D/n} )  & \hbox{if $i=0$ and $k=0$ and $j=1$ and $F(z)=Q(z)$}\\
 -(j-1) \operatorname{W}\big( - \frac{1}{j-1}
\big(\frac{n}{-jD}\big)^{1/(j-1)}  \big)  &
\hbox{if $i=0$ and $k=0$ and $j> 1$ and $F(z)=Q(z)$}\\
\end{cases}
		\end{equation}
\end{proposition}

\begin{proof}
If $i \geq 1$, then the saddle point location is the same in the
$P(z)$ and $Q(z)$ cases.
\begin{align}
\label{eq:Papprox}
	P(z) \sim \exp\left(
	\frac{A \log^{i-1} (-\log z)}{(-\log z)^2} \big(1+o(1)\big)
	\right).
\end{align}
The saddle is located at radius ``$r$'' with $r\frac{P'(z)}{P(z)}=n$, so using the approximation in \eqref{eq:Papprox} we get
\begin{align}\label{saddlecasea}
	\frac{P'(z)}{P(z)} \sim \frac{1}{z} \left( \frac{2A \log^{i-1}(-\log z)}{(-\log z)^3} - \frac{(i-1)A \log^{i-2}(-\log z)}{(-\log z)^3} \right).
\end{align}
If $i=1$, this simplifies to $2A/(-\log r)^3 = n$,
so the location of the saddle point~$r$ is:
\begin{align}
	r = \exp \bigg( - \bigg(  \frac{2A}{n}  \bigg)^{1/3}\bigg).
\end{align}
If $i>1$, because we only want a first-order solution
of~\eqref{saddlecasea}, we want to solve
\begin{align}
	\frac{2A \log^{i-1}(-\log r)}{(-\log r)^3} = n.
\end{align}
Taking a power of $1/(i-1)$ throughout yields
\begin{align}
	(2A )^{1/(i-1)}\log(-\log r) = n^{1/(i-1)}(-\log r)^{3/(i-1)}.
\end{align}
Taking logarithms and solving for $\log\log(-\log r)$ yields:
\begin{align}
	\log \log(-\log r) = \log\bigg(\frac{n}{2A }\bigg)^{1/(i-1)}
+ \frac{3}{i-1} \log(-\log r).
\end{align}
We know that $\log{x} = \log{b} + cx$ has solution
$x = -(1/c)\operatorname{W}(-bc)$.
Using $x = \log(-\log{r})$ 
and $b = \big(\frac{n}{2A }\big)^{1/(i-1)}$ and $c = 3/(i-1)$, it
follows that
\begin{align}
	\log(-\log{r}) = - \frac{i-1}{3} \operatorname{W}\left( - \frac{3}{i-1} \left(\frac{n}{2A }\right)^{1/(i-1)}  \right).
\end{align}
and it follows immediately that
\begin{align}
	r = \exp \bigg(
	-\exp \bigg( 
		- \frac{i-1}{3} \operatorname{W}\bigg( - \frac{3}{i-1} \bigg(\frac{n}{2A }\bigg)^{1/(i-1)}    \bigg) \bigg) \bigg).
\end{align}

If $i=0$ and $k \geq 1$, the saddle point location is the same in the
$P(z)$ and $Q(z)$ cases.
\begin{align}
\label{eq:Papprox2}
	P(z) \sim \exp\left(
	\frac{B \log^{k-1} (-\log z)}{-\log z} \big(1+o(1)\big)
	\right).
\end{align}
The saddle is located at radius ``$r$'' with $r\frac{P'(z)}{P(z)}=n$, so using the approximation in \eqref{eq:Papprox2} we get
\begin{align}\label{saddlecaseb}
	\frac{P'(z)}{P(z)} \sim \frac{1}{z} \left( \frac{B \log^{k-1}(-\log z)}{(-\log z)^2} - \frac{(k-1)B \log^{k-2}(-\log z)}{(-\log z)^2} \right).
\end{align}
If $k=1$, this simplifies to $B /(-\log r)^2 = n$,
so the location of the saddle point~$r$ is:
\begin{align}
	r = \exp \bigg( - \bigg(  \frac{B }{n}  \bigg)^{1/2}\bigg).
\end{align}
If $k>1$, because we only want a first-order solution
of~\eqref{saddlecaseb}, we want to solve
\begin{align}
	\frac{B \log^{k-1}(-\log r)}{(-\log r)^2} = n.
\end{align}
Taking a power of $1/(k-1)$ throughout yields
\begin{align}
	(B )^{1/(k-1)}\log(-\log r) = n^{1/(k-1)}(-\log r)^{2/(k-1)}.
\end{align}
Taking logarithms and solving for $\log\log(-\log r)$ yields:
\begin{align}
	\log\log(-\log r) =
        \log\bigg(\frac{n}{B}\bigg)^{1/(k-1)} + \frac{2}{k-1}\log(-\log r).
\end{align}
We know that $\log{x} = \log{b} + cx$ has solution
$x = -(1/c)\operatorname{W}(-bc)$.
Using $x = \log(-\log{r})$ 
and $b = \big(\frac{n}{B}\big)^{1/(k-1)}$ and $c = 2/(k-1)$, it
follows that
\begin{align}
	\log(-\log{r}) = - \left(\frac{k-1}{2}\right) \operatorname{W}\left( - \frac{2}{k-1} \left(\frac{n}{B}\right)^{1/(k-1)}  \right).
\end{align}
and it follows immediately that
\begin{align}
	r = \exp \bigg(
	-\exp \bigg( 
		- \bigg(\frac{k-1}{2}\bigg) \operatorname{W}\bigg( - \frac{2}{k-1} \bigg(\frac{n}{B}\bigg)^{1/(k-1)}    \bigg) \bigg) \bigg).
\end{align}

If $i=0$ and $j \geq 1$ and $k=0$, then
\begin{align}
\label{eq:Papprox3}
	P(z) &\sim \exp\left(
	C \log^{j+1} (-\log z) \big(1+o(1)\big)
	\right),\\
	Q(z) &\sim \exp\left(
	D \log^{j} (-\log z) \big(1+o(1)\big)
	\right).
\end{align}
The saddle is located at radius ``$r$'' with $r\frac{P'(z)}{P(z)}=n$, so using the approximation in \eqref{eq:Papprox3} we get
\begin{align}\label{saddlecasec}
	\frac{P'(z)}{P(z)} &\sim \frac{1}{z} \left( - \frac{(j+1)C \log^{j}(-\log z)}{-\log z} \right),\\
	\frac{Q'(z)}{Q(z)} &\sim \frac{1}{z} \left( - \frac{jD \log^{j-1}(-\log z)}{-\log z} \right).\label{saddlecased}
\end{align}
In the case $j=1$, when analyzing the saddle point for $Q(z)$,
we have simply
$\frac{Q'(z)}{Q(z)} \sim \frac{1}{z} \frac{D}{\log z}$, so 
the location of the saddle point $r$ satisfies
$\frac{D}{\log r} = n$, and thus $r = \exp(D/n)$.

If $j\geq 1$ (or if $j>1$ in the $Q$ case), 
because we only want a first-order solution
of~\eqref{saddlecasec} and~\eqref{saddlecased}, we want to solve
\begin{align}
	\frac{-(j+1)C \log^{j}(-\log r)}{-\log r} = n
\qquad\hbox{and}\qquad
	\frac{-jD \log^{j-1}(-\log r)}{-\log r} = n.
\end{align}
Taking a power of $1/j$ or $1/(j-1)$ (respectively) throughout yields
$$
	(-(j+1)C )^{1/j}\log(-\log r) = n^{1/j}(-\log r)^{1/j}
$$
and
$$
	(-jD )^{1/(j-1)}\log(-\log r) = n^{1/(j-1)}(-\log r)^{1/(j-1)}.
$$
Taking logarithms and solving for $\log\log(-\log r)$ yields:
$$
	\log\log(-\log r) 
= \log\bigg(\frac{n}{-(j+1)C }\bigg)^{1/j} + \frac{1}{j}\log(-\log r)
$$
and
$$
	\log\log(-\log r)
= \log\bigg(\frac{n}{-jD}\bigg)^{1/(j-1)} + \frac{1}{j-1}\log(-\log r).
$$
We know that $\log{x} = \log{b} + cx$ has solution
$x = -(1/c)\operatorname{W}(-bc)$.
Using $x = \log(-\log{r})$ 
and $b = \big(\frac{n}{-(j+1)C }\big)^{1/j}$ 
or, respectively, $b = \big(\frac{n}{-jD}\big)^{1/(j-1)}$,
and using 
$c = 1/j$ or, respectively, $c=1/(j-1)$, it
follows that the locations of the saddle point~$r$ satisfy, respectively:
\begin{align}
	\log(-\log{r}) = -j\!\ \operatorname{W}\left( - \frac{1}{j} \left(\frac{n}{-(j+1)C }\right)^{1/j}  \right).
\end{align}
and
\begin{align}
	\log(-\log{r}) = -(j-1)\!\ \operatorname{W}\left( 
- \frac{1}{j-1} \left(\frac{n}{-jD}\right)^{1/(j-1)}  \right).
\end{align}
and it follows immediately that
\begin{align}
	r = \exp \bigg(
	-\exp \bigg( 
		- j\!\ \operatorname{W}\bigg( - \frac{1}{j} \bigg(\frac{n}{-(j+1)C }\bigg)^{1/j}    \bigg) \bigg) \bigg)
\end{align}
and
\begin{align}
	r = \exp \bigg(
	-\exp \bigg( 
		- (j-1)\!\ \operatorname{W}\bigg( - \frac{1}{j-1}
                \bigg(\frac{n}{-jD}\bigg)^{1/(j-1)}    \bigg) \bigg) \bigg)
\end{align}

\end{proof}

\subsection{Central approximation}

We are using $f(z) = \log F(z) - n\log{z}$, as explained immediately
below equation~\eqref{eq:FExpansion}.  We have
\begin{align}\label{unifiedframework}
[z^{n}]F(z) &\sim \frac{1}{2\pi}F(r)r^{-n}\int\limits_{\substack{-\pi \\ -\infty}}^{\substack{\infty \\ \pi}}
			\exp\left(
				\frac{1}{2}f''(r)(re^{i\theta}-r)^2
			\right)d\theta\\
\label{saddlepointequation}	&\sim \frac{1}{2\pi}\exp\left(
	\frac{a\big(\alpha(n)\big)}{\exp(2\alpha(n))} + \frac{b\big(\alpha(n)\big)}{\exp(\alpha(n))} + g\big(\alpha(n)\big) + h\big(\exp(\alpha(n))\big)
	\right)e^{n\exp(\alpha(n))}\sqrt\frac{-2\pi}{f''(r)r^{2} }
\end{align}

It follows from Proposition~\ref{prop:saddle} that
\begin{equation}\label{nexpalphan}
n\exp(\alpha(n)) \sim \begin{cases}
(2A)^{1/3}n^{2/3}
& \hbox{if $i=1$}\\
n\exp\big(
- \frac{i-1}{3} \operatorname{W}\big( - \frac{3}{i-1} \big(\frac{n}{2A }\big)^{1/(i-1)}    \big)
\big)
& \hbox{if $i>1$}\\
 (Bn)^{1/2} & \hbox{if $i=0$, $k=1$}\\
 n\exp\big(-\frac{k-1}{2} \operatorname{W}\big( - \frac{2}{k-1}
 \big(\frac{n}{B}\big)^{1/(k-1)}    \big)\big)  & \hbox{if $i=0$, $k > 1$}\\
 n\exp\big(-j\!\ \operatorname{W}\big( - \frac{1}{j} \big(\frac{n}{-(j+1)C }\big)^{1/j}    \big)\big)  &
\hbox{if $i=0$, $k=0$, $j\geq 1$, $F(z)=P(z)$}\\
-D  & \hbox{if $i=0$, $k=0$, $j=1$, $F(z)=Q(z)$}\\
n\exp\big(-(j-1) \operatorname{W}\big( - \frac{1}{j-1}
\big(\frac{n}{-jD }\big)^{1/(j-1)}  \big)\big)  &
\hbox{if $i=0$, $k=0$, $j> 1$, $F(z)=Q(z)$}\\
\end{cases}
\end{equation}

It also follows from Proposition~\ref{prop:saddle} that
$$\log\Big(\sqrt{\frac{1}{r^{2}}}\!\ \Big) = -\log(r) = \exp(\alpha(n))$$
so
\begin{equation}\label{logsqrt1overrsquared}
\log\Big(\sqrt{\frac{1}{r^{2}}}\!\ \Big)
 \sim \begin{cases}
(2A)^{1/3}n^{-1/3}
& \hbox{if $i=1$}\\
\exp\big(
- \frac{i-1}{3} \operatorname{W}\big( - \frac{3}{i-1} \big(\frac{n}{2A }\big)^{1/(i-1)}    \big)
\big)
& \hbox{if $i>1$}\\
 B^{1/2}n^{-1/2} & \hbox{if $i=0$, $k=1$}\\
 \exp\big(-\frac{k-1}{2} \operatorname{W}\big( - \frac{2}{k-1}
 \big(\frac{n}{B}\big)^{1/(k-1)}    \big)\big)  & \hbox{if $i=0$, $k > 1$}\\
 \exp\big(-j\!\ \operatorname{W}\big( - \frac{1}{j} \big(\frac{n}{-(j+1)C }\big)^{1/j}    \big)\big)  &
\hbox{if $i=0$, $k=0$, $j\geq 1$, $F(z)=P(z)$}\\
-D/n  & \hbox{if $i=0$, $k=0$, $j=1$, $F(z)=Q(z)$}\\
\exp\big(-(j-1) \operatorname{W}\big( - \frac{1}{j-1}
\big(\frac{n}{-jD }\big)^{1/(j-1)}  \big)\big)  &
\hbox{if $i=0$, $k=0$, $j> 1$, $F(z)=Q(z)$}\\
\end{cases}
\end{equation}

We also use Proposition~\ref{prop:saddle} to compute 
\begin{align}\label{aoverexp2alphanetc}
\frac{a\big(\alpha(n)\big)}{\exp(2\alpha(n))}
&\sim \begin{cases}
(\frac{1}{2})(2A)^{1/3}n^{2/3}
& \hbox{if $i=1$}\\
A\big( - \frac{i-1}{3} \operatorname{W}\big( - \frac{3}{i-1} \big(\frac{n}{2A }\big)^{1/(i-1)}    \big)\big)^{i-1}\exp\big(
2\big(\frac{i-1}{3}\big) \operatorname{W}\big( - \frac{3}{i-1} \big(\frac{n}{2A }\big)^{1/(i-1)}    \big)
\big)
& \hbox{if $i>1$}
\end{cases}\\
\frac{b\big(\alpha(n)\big)}{\exp(\alpha(n))}
&\sim \begin{cases}
(Bn)^{1/2} & \hbox{if $i=0$, $k=1$}\\
B\big(  - \frac{k-1}{2} \operatorname{W}\big( - \frac{2}{k-1}
 \big(\frac{n}{B}\big)^{1/(k-1)}    \big) \big)^{k-1}\exp\big(\frac{k-1}{2} \operatorname{W}\big( - \frac{2}{k-1}
 \big(\frac{n}{B}\big)^{1/(k-1)}    \big)\big)  & \hbox{if $i=0$, $k > 1$}
\end{cases}\\
g\big(\alpha(n)\big)
&\sim \begin{cases}
C\big(  -j\!\ \operatorname{W}\big( - \frac{1}{j} \big(\frac{n}{-(j+1)C
}\big)^{1/j}    \big) \big)^{j+1} & \hbox{if $i=0, k=0, j\geq 1$, $F(z) = P(z)$}\\
D\log( {-D/n} ) & \hbox{if $i=0, k=0, j= 1$, $F(z) = P(z)$}\\
D\big(  -(j-1) \operatorname{W}\big( - \frac{1}{j-1}
\big(\frac{n}{-jD}\big)^{1/(j-1)}  \big) \big)^{j} & \hbox{if $i=0, k=0, j> 1$, $F(z) = P(z)$}
\end{cases}
\end{align}

Finally, we compute $f^{\prime\prime}$,
using $f(z)=\log\big(F(z)\big)-n\log(z)$.
\begin{proposition}
From~\eqref{eq:FExpansion}, for $i > 1$, we have
$f^{\prime\prime}(r) = 6A\alpha(n)^{i - 1}e^{-4\alpha(n)}
+ \Theta(\alpha(n)^{i - 2}e^{-4\alpha(n)}  )$,
or in the case $i=1$, we have the same first-order approximation, but
a slightly different error term, namely 
$f^{\prime\prime}(r) = 6Ae^{-4\alpha(n)}
+ \Theta(e^{-3\alpha(n)}  )$.

For $i=0$ and $k > 1$, we have
$f^{\prime\prime}(r) = -2B\alpha(n)^{k - 1}e^{-3\alpha(n)}
+ \Theta(\alpha(n)^{k-2}e^{-3\alpha(n)}  )$,
or, very similarly, in the $k=1$ case, we have 
$f^{\prime\prime}(r) = -2Be^{-3\alpha(n)} + \Theta(e^{-2\alpha(n)}  )$.

For $i=0$ and $k=0$ and $j \geq 1$, in the $P(z)$ case, we have
$f^{\prime\prime}(r) = -(j+1)C\alpha(n)^{j}e^{-2\alpha(n)}
+ \Theta(\alpha(n)^{j-1}e^{-2\alpha(n)}  )$.

For $i=0$ and $k=0$ and $j > 1$, in the $Q(z)$ case, we have
$f^{\prime\prime}(r) = -jD\alpha(n)^{j-1}e^{-2\alpha(n)}
+ \Theta(\alpha(n)^{j-2}e^{-2\alpha(n)}  )$,
or, very similarly, in the $j=1$ case, we have
$f^{\prime\prime}(r) = -D\!\ e^{-2\alpha(n)}
+ \Theta(e^{-\alpha(n)}  )$.
\end{proposition}

It follows that
\begin{equation}\label{logsqrtoneoverfdoubleprime}
\log\Big(\sqrt{\frac{1}{f^{\prime\prime}(r)}}\!\ \Big)
\sim \begin{cases}
2\alpha(n)
& \hbox{if $i\geq 1$}\\
\frac{3}{2}\alpha(n)
& \hbox{if $i=0$, $k\geq 1$}\\
\alpha(n)  &
\hbox{if $i=0$, $k=0$, $j\geq 1$, $F(z)=P(z)$}\\
\alpha(n) & \hbox{if $i=0$, $k=0$, $j\geq 1$, $F(z)=Q(z)$}
\end{cases}
\end{equation}

\begin{theorem}\label{theorem63}
Let $F(z)$ denote $P(z)$ or $Q(z)$.
We have
\begin{equation}
\log{[z^{n}]F(z)} \sim \begin{cases}
(3/2)(-3)^{-(i-1)/3}(2A)^{1/3}(\ln{n})^{(i-1)/3}n^{2/3}
& \hbox{if $i\geq 1$}\\
(2)(-2)^{-(k-1)/2}B^{1/2}(\ln{n})^{(k-1)/2}n^{1/2} & \hbox{if $i=0$, $k\geq 1$}\\
(-1)^{j+1}(C)(\ln{n})^{j+1} &
\hbox{if $i=0$, $k=0$, $j\geq 1$, $F(z)=P(z)$}\\
(-1)^{j}(D)(\ln{n})^{j} & \hbox{if $i=0$, $k=0$, $j\geq 1$, $F(z)=Q(z)$}
\end{cases}
\end{equation}
where, as defined in~\eqref{definitionofAconstant}, we have
\begin{equation}
A = \begin{cases}\frac{(-1)^{i-1}\zeta(3)^{j+1}\pi^{2k}}{6^{k}(i-1)!}
  & \hbox{in the $F(z)=P(z)$ case}\\
\frac{(3)
  (-1)^{i-1}\zeta(3)^{j+1}\pi^{2k}}{(4)(6^{k})(i-1)!}
  & \hbox{in the $F(z)=Q(z)$ case}
\end{cases}
\end{equation}
and as defined in~\eqref{definitionofBconstant}, we have
\begin{equation}
B = \begin{cases}
\frac{(-1)^{i+k-1}\pi^{2(j+1)}}{6^{j+1}2^{i}(k-1)!}
  & \hbox{in the $F(z)=P(z)$ case}\\
\frac{(-1)^{i+k-1}\pi^{2(j+1)}}{6^{j+1}2^{i+1}(k-1)!}
  & \hbox{in the $F(z)=Q(z)$ case}\\
\end{cases}
\end{equation}
and, finally, as in~\eqref{definitionofCconstant}
and~\eqref{definitionofDconstant} respectively, we have
\begin{equation}
C = \frac{(-1)^{i+j+k+1}}{2^{k}(j+1)!\!\ 12^{i}}\quad\hbox{in the
  $F(z) = P(z)$ case, and }
D = \ln(2)\frac{(-1)^{i+j+k}}{2^{k}j!\!\ 12^{i}}\quad\hbox{in the
  $F(z) = Q(z)$ case}
\end{equation}

\end{theorem}

\begin{proof}
We are using the framework from equation~\eqref{saddlepointequation}.
Assembling the contributions from~\eqref{nexpalphan}
and~\eqref{aoverexp2alphanetc}, and 
noting that~\eqref{logsqrt1overrsquared} and~\eqref{logsqrtoneoverfdoubleprime}
do not contribute to the first-order asymptotics of the logarithm of
the coefficients, we put these contributions together and in~\eqref{saddlepointequation},
taking the logarithm, we obtain these first order approximations:
\begin{equation}
\log{[z^{n}]F(z)} \sim \begin{cases}
(\frac{3}{2})(2A)^{1/3}n^{2/3}
& \hbox{if $i=1$}\\
n\exp\big(
- \frac{i-1}{3} \operatorname{W}\big( - \frac{3}{i-1} \big(\frac{n}{2A }\big)^{1/(i-1)}    \big)
\big)\\
{}\qquad+ A\big( - \frac{i-1}{3} \operatorname{W}\big( - \frac{3}{i-1}
\big(\frac{n}{2A }\big)^{1/(i-1)}    \big)\big)^{i-1}\\
\qquad\qquad{}\times\exp\big(
2\big(\frac{i-1}{3}\big) \operatorname{W}\big( - \frac{3}{i-1} \big(\frac{n}{2A }\big)^{1/(i-1)}    \big)
\big)
& \hbox{if $i>1$}\\
2(Bn)^{1/2} & \hbox{if $i=0$, $k=1$}\\
n\exp\big(-\frac{k-1}{2} \operatorname{W}\big( - \frac{2}{k-1}
 \big(\frac{n}{B}\big)^{1/(k-1)}    \big)\big)\\
{}\qquad+ B\big(  - \frac{k-1}{2} \operatorname{W}\big( - \frac{2}{k-1}
 \big(\frac{n}{B}\big)^{1/(k-1)}    \big) \big)^{k-1}\\
\qquad\qquad{}\times\exp\big(\frac{k-1}{2} \operatorname{W}\big( - \frac{2}{k-1}
 \big(\frac{n}{B}\big)^{1/(k-1)}    \big)\big)  & \hbox{if $i=0$, $k > 1$}\\
 n\exp\big(-j\!\ \operatorname{W}\big( - \frac{1}{j}
 \big(\frac{n}{-(j+1)C }\big)^{1/j}    \big)\big)\\
{}\qquad+ C\big(  -j\!\ \operatorname{W}\big( - \frac{1}{j} \big(\frac{n}{-(j+1)C
}\big)^{1/j}    \big) \big)^{j+1} &
\hbox{if $i=0$, $k=0$, $j\geq 1$, $F(z)=P(z)$}\\
-D + D\log( {-D/n} ) & \hbox{if $i=0$, $k=0$, $j=1$, $F(z)=Q(z)$}\\
n\exp\big(-(j-1) \operatorname{W}\big( - \frac{1}{j-1}
\big(\frac{n}{-jD }\big)^{1/(j-1)}  \big)\big)\\
{}\qquad+ D\big(  -(j-1) \operatorname{W}\big( - \frac{1}{j-1}
\big(\frac{n}{-jD}\big)^{1/(j-1)}  \big) \big)^{j} &
\hbox{if $i=0$, $k=0$, $j> 1$, $F(z)=Q(z)$}\\
\end{cases}
\end{equation}
Now we  use 
$\operatorname{W}(z) = \ln(z) - \ln\ln(z) + o(1)$ to simplify everything.
The case for $i>1$ simplifies to 
$(3/2)(-3)^{-(i-1)/3}(2A)^{1/3}(\ln{n})^{(i-1)/3}n^{2/3}$,
which agrees with the case $i=1$.
The case for $i=0$, $k>1$ simplifies to
$(2)(-2)^{-(k-1)/2}B^{1/2}(\ln{n})^{(k-1)/2}n^{1/2}$,
which agrees with the case $i=0$, $k=1$.
The case for $i=0$, $k=0$, $j\geq 1$, with $F(z)=P(z)$, simplifies to
$(-1)^{j+1}(j+1)(C)(\ln{n})^{j} + (-1)^{j+1}(C)(\ln{n})^{j+1}$,
which then simplifies to $(-1)^{j+1}(C)(\ln{n})^{j+1}$.
The case for $i=0$, $k=0$, $j > 1$, with $F(z)=Q(z)$, simplifies to
$(-1)^{j}(j)(D)(\ln{n})^{j-1} + (-1)^{j}(D)(\ln{n})^{j}$,
which further simplifies to $(-1)^{j}(D)(\ln{n})^{j}$, which agrees
with the case $i=0$, $k=0$, $j=1$, with $F(z)=Q(z)$.
Theorem~\ref{theorem63} follows as a result of these simplifications.
\end{proof}

% ***************************************************************
%                                 BIBLIOGRAPHY
% ***************************************************************

\section*{Acknowledgements}

R. G\'omez is supported by grants DGAPA-PAPIIT IN107718 and IN110221.

M.D.~Ward's research is supported by National Science Foundation (NSF)
grants 0939370, 1246818, 2005632, 2123321, 2118329, and 2235473 by the
Foundation for Food and Agriculture Research (FFAR) grant 534662, by
the National 
Institute of Food and Agriculture (NIFA) grants 2019-67032-29077,
2020-70003-32299, 2021-38420-34943, and 2022-67021-37022
by the Society Of Actuaries grant 19111857, by
Cummins Inc., by Gro Master, by Lilly Endowment, and by Sandia National Laboratories.

\end{document}